\documentclass[review,onefignum,onetabnum]{siamart190516}

\usepackage{lipsum}
\usepackage{amsmath,bm}
\usepackage{amsfonts}
\usepackage{graphicx}
\usepackage{psfrag}
\usepackage{tikz}
\usetikzlibrary{calc,positioning}
\definecolor{corange}{rgb}{0.93, 0.57, 0.13}
\usepackage{simplewick}
\usepackage{mathrsfs}
\usepackage{url,hyperref}
\usepackage{setspace}

\usepackage{algorithm}
\usepackage{algorithmic}

\usepackage{amssymb}
\usepackage{booktabs}
\usepackage{comment}
\usepackage{subfigure}
\usepackage{float}
\usepackage{cite}
\usepackage{epstopdf}
\newtheorem{exa}{\bf Example}

\newcommand{\bs}[1]{\boldsymbol{#1}}
\allowdisplaybreaks[4]

\newtheorem{lem}{Lemma}[section]

\def \bx{\bs x}

\def \by{\bs y}

\def\B{\mathbb{B}}

\def\E{\mathbb{E}}
\def\tphi{\Phi^{\ast}}
\def\hphi{\widehat{\Phi}}
\def\Q{\widetilde{Q}}

\def\d{\mathrm{d}}

\def\V{\mathrm{Var}}
\def\oa{\omega^{\frac{\alpha}{2}}}
\def\oaa{\omega^{-\frac{\alpha}{2}}}
\def\Ia{\mathcal{I}_{N_x}^{\frac{\alpha}2}}
\def\lja{l^{\frac{\alpha}{2}}_j(x)}

\def\Ja{\mathcal{J}^{\frac{\alpha}2}}
\def\fna {\mathcal{F}^{\frac{\alpha}{2}}_{N_x}}
\def\ina {\mathcal{I}_{N_x}^{\frac{\alpha}2}}

\newcommand \dint {\displaystyle\int}

\usepackage{lipsum}
\usepackage{amsfonts}
\usepackage{graphicx}
\usepackage{epstopdf}
\usepackage{algorithmic}
\usepackage{lineno}
\ifpdf
\DeclareGraphicsExtensions{.eps,.pdf,.png,.jpg}
\else
\DeclareGraphicsExtensions{.eps}
\fi

\newsiamremark{remark}{Remark}
\newsiamremark{hypothesis}{Hypothesis}
\crefname{hypothesis}{Hypothesis}{Hypotheses}
\newsiamthm{claim}{Claim}

\headers{Spectral Monte Carlo Method for linear PDEs }{J. Feng, C. Sheng, and  C. Xu}

\title{Exponentially accurate Spectral Monte Carlo Method for linear PDEs and their error estimates \thanks{Submitted to the editors DATE.
\funding{The research of the second author is partially supported by the National Natural Science Foundation of China (Nos. 12201385 and 12271365) and the Fundamental Research Funds for the Central Universities 2021110474. The research of the third author is partially supported by the Fundamental Research Funds for the Central Universities. ${}^{\star}$Corresponding author. }}}

\author{Jiaying Feng \thanks{School of Mathematics, Shanghai University of Finance and Economics, Shanghai 200433, China. Email: \email{fengjiaying@163.sufe.edu.cn} (J. Feng); \email{ctsheng@sufe.edu.cn} (C. Sheng); \email{xu.chenglong@shufe.edu.cn} (C. Xu).  }
\and  Changtao Sheng$^{\dagger,\star}$\and Chenglong Xu$^{\dagger,\star}$ }

\usepackage{amsopn}

\ifpdf
\hypersetup{
pdftitle={Exponentially accurate Spectral Monte Carlo Method for linear PDEs and their error estimates},
pdfauthor={J. Feng, C. Sheng, and  C. Xu}
}
\fi

\begin{document}
\nolinenumbers

\maketitle
\begin{abstract}
This paper introduces a spectral Monte Carlo iterative method (SMC) for solving linear Poisson and parabolic equations driven by $\alpha$-stable Lévy process with $\alpha\in (0,2)$, which was initially proposed and developed by Gobet and Maire in their pioneering works (Monte Carlo Methods Appl 10(3-4), 275--285, 2004, and SIAM J Numer Anal 43(3), 1256--1275, 2005) for the case $\alpha=2$. 
The novel method effectively integrates multiple computational techniques, including the interpolation based on generalized Jacobi functions (GJFs), space-time spectral methods, control variates techniques, and a novel walk-on-sphere method (WOS). The exponential convergence of the error bounds is rigorously established through finite iterations for both Poisson and parabolic equations involving the integral fractional Laplacian operator. Remarkably, the proposed space-time spectral Monte Carlo method (ST-SMC) for the parabolic equation is unified for both $\alpha\in(0,2)$ and $\alpha=2$. 
Extensive numerical results are provided to demonstrate the spectral accuracy and efficiency of the proposed method, thereby validating the theoretical findings. 
\end{abstract}



\section{Introduction}

Spectral methods have gained significant popularity in computational fluid dynamics, quantum mechanics, stochastic differential equations, uncertainty quantification, and more (see, e.g., \cite{Canuto2012,LeMaitre2010,Lord2014,Shen2006,Karniadakis2013,Xiu2010}), primarily due to their ability to provide highly accurate approximations for the solutions of PDEs, as long as these solutions are sufficiently smooth. This is one of the main appealing advantages of spectral methods, the so-called ``spectral accuracy", which is also the goal at the algorithmic level in this paper. 

Stochastic simulation plays a crucial role in understanding complex real-world phenomena, such as options pricing, turbulence, multiphase flows, molecular dynamics simulations, crystal plasticity, radiation transport, radiotherapy, uncertainty quantification and more \cite{Cai2011,Nanbu1980,LeMaitre2010,Bird1996,Glasserman2004,Holmes1996,Ladd1994,Xiu2002,Elfverson2016}, because it encompasses various advantages not typically found in deterministic methods, including being well-suited for solving high-dimensional problems, handling PDEs on complex geometric domains, facilitating parallel computation, and being simpler to implement. The theoretical basis for stochastic methods in solving PDEs resides in the exploitation of a probabilistic representation, namely, the renowned Feynman-Kac formula, which expresses the solution of PDEs as an expectation form. With this, various sampling-based stochastic methods such as the binomial tree method, the multi-tree method, etc., can obtain numerical solutions by directly simulating stochastic processes \cite{Bally2005,Babuska2004,Glasserman2004,Lyons2016,Shao2015,Shao2020}. Unlike direct simulation methods, the walk-on-spheres method employs transition probabilities to determine the next position of the sphere, thereby simulating stochastic processes on the sphere instead of complete trajectories \cite{Muller1956,DeLaurentis1990,Kyprianou2018,Mascagni2004,Yan2021,Ding2022,Ding2023}. This approach notably reduces computational costs and memory demands and offers an efficient and flexible approach.
However, these stochastic methods usually achieve only half-order accuracy, i.e., $\mathcal{O}(N^{-1/2})$, where  $N$ denotes the number of samples. Although it is possible to improve the accuracy to first-order by making specific modifications, achieving higher-order accuracy becomes extremely challenging.

Many efforts have been made to improve the accuracy of stochastic methods for PDEs based on solving backward stochastic differential equations, for instance, the multistep method \cite{Chassagneux2012,GobetMenozzi2010,Han2022,Zhao2010,Komori2017,Brehier2024}, the adapted control variates\cite{Alanko2013,Hutzenthaler2021,Gobet2010,Takahashi2022} and so on. 
Besides, various kinds of higher-order numerical methods, normally second order, have been developed to solve the stochastic PDEs, where high-order typically refers to methods beyond first-order accuracy (cf.\cite{Andersson2019,Xiu2005,Zhao2014,During2012}). However, it is not trivial to directly extend these high-order techniques for solving stochastic PDEs to the approach of simulating the path of a stochastic process that is of interest to this paper.

Gobet and Maire  \cite{Gobet2004} introduced for the first time a spectral Monte Carlo method for the Poisson equation, which combined the idea of iterative variance reduction techniques from sequential Monte Carlo methods with the spectral collocation methods to achieve the Monte Carlo method with spectral accuracy. Subsequently, they updated the proposed method from the perspective of the sequential control variates algorithm, extended it to linear parabolic equations, and proved that the bias and variance of the proposed method decrease geometrically with respect to the number of algorithm steps (cf. \cite{Gobet2005}). The key aspect of those methods lies in the precise reconstruction of the numerical solution from rough data during the iterative process, as well as the terms with different order derivatives of numerical solutions in the governing equation. This accurate reconstruction is achieved by using the connection relationship between derivatives of orthogonal basis functions and their coefficients, thus enabling the SMC method to obtain appealing spectral accuracy.  Later, these methods or other similar approaches were further extended to various problems (see, e.g., \cite{Gobet2010,GobetMenozzi2010,Maire2015,Billaud-Friess2024} and the references therein). 
We also remark that various techniques for improving sampling efficiency in Monte Carlo methods have been discussed in some recent works, interested readers may refer to cubic stratification \cite{Chopin2024} and the Quasi-Monte Carlo method \cite{Longo2021}.

On the other hand, some recent studies on stochastic algorithms for PDEs with integral fractional Laplacian operator (IFL), where the IFL with index $\alpha\in(0,2)$ is defined via the hypersingular integral representation (cf.\! \cite{Nezza2012}):
\begin{equation}\label{fracLap-defn}
(-\Delta)^{\frac{\alpha}2} u(\bx)=C_{d,\alpha}\, {\rm p.v.}\! \int_{\mathbb R^d} \frac{u(\bx)-u(\by)}{|\bx-\by|^{d+\alpha}}\, {\rm d}\by,\quad
C_{d,\alpha} :=\frac{2^{\alpha}\Gamma(\frac{d+\alpha}{2})}{\pi^{\frac{d}2}\Gamma(1-\frac{\alpha}2)},
\end{equation}
with ``p.v." stands for the principal value and $C_{d,\alpha} $ is the normalization constant.  Note that when $\alpha=2$ the operator coincides with the usual negative Laplacian $-\Delta$. The non-local and singular nature of IFL poses major difficulties in discretization and analysis for the deterministic numerical methods (cf. \cite{Acosta2017,Du2012,Huang2014,Wu2020} and the references therein). 
On the contrary, stochastic methods have demonstrated remarkable success in numerical solving this problem, especially in the case of high-dimensions and complex domains. 
Kyprianou et al. \cite{Kyprianou2018} proposed an unbiased walk-on-sphere method, based on the Feynman-Kac formula, for solving the fractional Poisson equation in two dimensions. Jiao et al. \cite{Jiao2023} proposed a modified version of the walk-on-sphere method for solving fractional Poisson equations in high-dimensional balls. Sheng et al. \cite{Sheng2023} recently developed an efficient Monte Carlo method for solving fractional PDEs on irregular domains in high dimensions, and also conducted a rigorous error analysis for the proposed scheme. 

As far as we know, the following questions are still unsolved in the development of
spectral Monte Carlo methods for the linear PDEs: (i).\,extend spectral Monte Carlo method for the Poisson equation involving IFL; (ii).\,extend space-time spectral Monte Carlo method for the parabolic equation involving IFL. However, the extension of the spectral Monte Carlo methods to the Poisson and parabolic equations driven by the $\alpha$-stable Lévy process is still challenging. The main difficulty in developing the spectral Monte Carlo method for PDEs involving IFL lies in precisely reconstructing the numerical solution and its fractional derivatives from the inaccurate data in each iteration step, that is, the stochastic solution obtained by the WOS methods.

\begin{figure}[htbp]
\centering  
\subfigure 
{ \includegraphics[width=8.8cm,height=5.6cm]{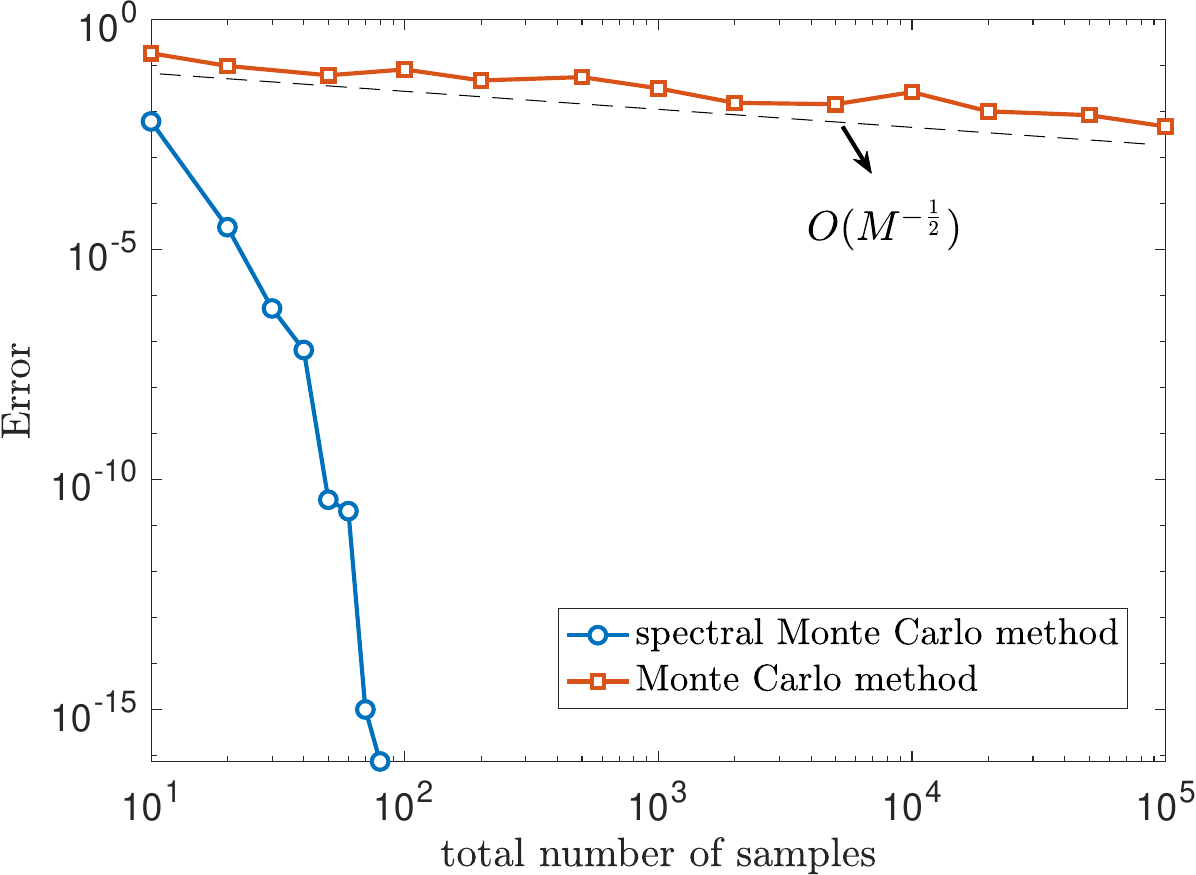}}
\caption{A comparison of the convergence rates between the Monte Carlo method and the spectral Monte Carlo method for solving the fractional Poisson equation $(-\Delta)^{\frac{\alpha}{2}}u(x)=f(x) \;{\rm in}\,\Omega$; $u(x)=0\;{\rm on}\,\Omega^c.$ Here, we take the exact solution $u(x)=(1-x^2)^{\frac{\alpha}{2}}(x^2+x+1)$, with $\alpha=0.4$ (cf. Example \ref{Ex:1}). It is observed that SMC rapidly achieves machine-level precision, whereas MC struggles to reach high accuracy.}\label{fig_MCvsSMC}
\end{figure}

The main purpose of this paper is to construct a new spectral Monte Carlo method for solving the Poisson equation driven by $\alpha$-stable  L$\acute{\text{e}}$vy process, to prove rigorously its error estimates, and to extend to the space-time spectral Monte Carlo method for parabolic equation driven by $\alpha$-stable  L$\acute{\text{e}}$vy process as well as its error estimate. Fig \ref{fig_MCvsSMC} illustrates a comparison of the convergence rates between the Monte Carlo method and the spectral Monte Carlo method (cf. Example \ref{Ex:1} in section \ref{sect5} for more detail).
Our main contributions include the following:
\begin{itemize}
\item The proposed method not only maintains spectral accuracy, as with traditional spectral methods, but also brings a delightful surprise: the value of the numerical solution at each mesh point can be obtained independently at each iteration process. Thus, it obviates the requirement for solving linear systems, thereby simplifying and accelerating the computation via parallel computation. This advantage becomes particularly prominent in space-time methods, as it enables parallel computation for each individual spatiotemporal node.

\item For the Poisson equation driven by $\alpha$-stable  L$\acute{\text{e}}$vy process, I have extended the algorithm from the case $\alpha=2$ (cf. \cite{Gobet2005}) to a more general model $\alpha\in(0,2)$ by utilizing the interpolation based on generalized Jacobi basis functions to accurately reconstruct the numerical solution at any points. 

\item As far as we know, we explore, for the first time, the Monte Carlo methods with spectral accuracy in both spatial and temporal for solving the parabolic equations driven by $\alpha$-stable  L$\acute{\text{e}}$vy process with $\alpha\in(0,2]$, i.e., thereby including the classical and fractional parabolic equations. The existing works (cf.\cite{Gobet2005}), in contrast, only focused on addressing Poisson equations or employing space-time methods based on piecewise linear interpolation in time for solving parabolic equations with $\alpha=2$.
\end{itemize}

The rest of the paper is organized as follows. In Section \ref{sect2}, we develop the spectral Monte Carlo method for solving the Poisson equation driven by $\alpha$-stable  L$\acute{\text{e}}$vy process. The corresponding bias 
of the proposed method are presented in Section \ref{sect3}.  
The extension to space-time spectral Monte Carlo method for parabolic equation as well as its bias is presented in Section \ref{sect4}. Finally, in Section \ref{sect5}, we present
several numerical examples to illustrate the high accuracy and efficiency of the proposed spectral Monte Carlo method for the PDEs in both fractional and classical cases.

\section{Spectral Monte Carlo method for Poisson-type equation in 1D} \label{sect2}
\setcounter{lem}{0} \setcounter{thm}{0}  \setcounter{rem}{0}
In this section, we propose an exponentially accurate Monte Carlo algorithm for the approximation of the Poisson equation in both classical and fractional cases, which draws inspiration from the spectral Monte Carlo method (cf. \cite{Gobet2004,Gobet2005}) for classical Poisson equations. 
Before that, we introduce the Monte Carlo method and review some relevant properties of generalized Jacobi polynomials (GJFs), which will be an indispensable building block for the construction of our novel spectral Monte Carlo method. 

\subsection{WOS method}
We consider the following Poisson equation involving IFL with the Dirichlet boundary condition and $\alpha\in(0,2)$:
\begin{equation}\label{uf}
\begin{cases}
(-\Delta)^{\frac{\alpha}2}u(x)=f(x),\;&x \in \ \Omega,\\[4pt]
u(x)=g(x), &x\in \ \Omega^c,
\end{cases}
\end{equation}
where  $\Omega$ is an open bounded domain and  $f: \Omega\to \mathbb R$ is a given function in a suitable space. 
In what follows, we will introduce the WOS method for the solutions of Poisson equation \eqref{uf} in \cite{Sheng2023}.  To this end, we first recall the Feynman-Kac formula that the solution of \eqref {uf} associated with $\alpha$-stable L$\acute{\text{e}}$vy process $X^{\alpha}_{t} $ and can be rewritten in the form of path integral (cf. \cite{Kyprianou2018})
\begin{equation} \label{fkx}
u(\bx) = \mathbb{E}_{X_{0}^{\alpha }=\bx} \Big[g(X_{\tau_{\Omega}^{\bx}}^{\alpha})+
\int_{0}^{\tau_{\Omega}^{\bx}}f(X^{\alpha}_{t})\,{\rm d}t\Big] :=\mathbb{E}_{\bx}\big[\Phi_{\bx}(f,g,X^{\alpha}_{t})\big],\;\; \bx\in\Omega,
\end{equation}
where  $\tau_{\Omega}^{\bx}$ is the first exit to domain $\Omega$ and $\mathbb{E}_{\bx}$ is an abbreviation for $\mathbb{E}_{X_{0}^{\alpha }=\bx} $.
Instead of simulating the irregular trajectories of symmetric $\alpha$-stable  L$\acute{\text{e}}$vy processes, the Monte Carlo method  takes advantage of the explicit expressions of $P_{r}$ as the probability density functions, which gives the transition probability of the symmetric $\alpha$-stable process trajectory inside the ball during its passage from the center to the boundary $\partial \Omega$ or outside the domain $\Omega$ (cf. Fig.\,\ref{ballregion}). 
Since the Markov process will leave the given domain at a finite time, then we denote the stopping step for the random walk $ X^\alpha_{m}$ by $L = \text{inf} \{ m : X^\alpha_{m}\notin \Omega \}.$

\begin{figure}[!ht]
\centering  
\subfigure
{
\includegraphics[width=5.3cm,height=4.3cm]{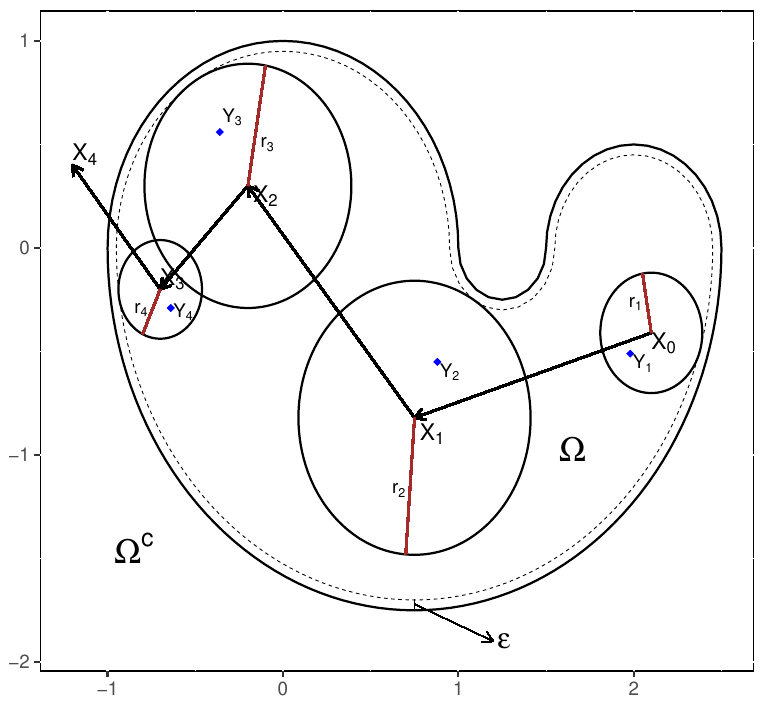}}
\caption{The path of walk on $2$-D irregular domain}\label{ballregion}
\end{figure}

According to \cite[Corollary 2.1]{Sheng2023}, we can employ a walk-on-spheres based Monte Carlo method to approximate \eqref{fkx}. 
More precisely,  let $r > 0$ and $\alpha\in(0,2]$, and assume that $f\in L^{1}(\Omega)\cap C(\overline{\Omega})$, $g\in L^{1}_{\alpha}(\Omega^{c})\cap C(\Omega^{c}) $, then the numerical solution $u_\ast(\bx)$ for \eqref{uf} can be evaluated by
\begin{equation}\label{iresolu}
\begin{split}
&u_\ast(\bx)= \frac{1}{M}\sum_{i=1}^M \Big[g(X_{L}^{i})+\sum_{\ell=0}^{L-1}\zeta(X_\ell^i)\mathbb{E}^\ast_{\widetilde{\mathbb{Q}}}[f(Y_{\ell+1}^i)]\Big]:=
 \frac{1}{M}\sum_{i=1}^M \Big[ \tphi_{\bx}(f,g,X_{t}^{i})\Big],
\end{split}
\end{equation}
where $X_\ell^i:=X^{\alpha,i}_\ell,$ $0\leq \ell\leq L$, $X_0=\bx\in \Omega$, and 
\begin{equation}\label{zetaE}
\zeta(\bx) = \int_{\B_{r}^n}Q
(\bx,\by)\,\d \by,\;\;\;\mathbb{E}^\ast_{\widetilde{\mathbb{Q}}}[f(Y_{\ell+1}^i)] = \frac{1}{M_1}\sum_{j=1}^{M_1}f(Y_{\ell+1}^{i,j}).
\end{equation} 
In the above, $\widetilde{Q}(\bm{x},\bm{y})= \frac{Q(\bm{x},\bm{y})}{\int_{\mathbb{B}_r^n}Q(\bm{x},\bm{y})\,{\rm d}\bm{y}}$, and $Q(\bm{x},\bm{y})$ denote the Green's function for $\bm{x}\neq \bm{y}$ is given by
\begin{equation}
Q(\bm{x},\bm{y})=
\begin{cases}
\hat{C}_n^\alpha \big|\bm{y}-\bm{x}\big|^{\alpha-n} \int_{0}^{\rho(\bm{x},\bm{y})} \frac{t^{\frac{\alpha}{2}-1}}{(t+1)^2}\,{\rm d}t,\;\;\;&\alpha \neq n,\\[6pt]
\hat{C}_1^{\frac{1}{2}}\log\big(\frac{r^2-\bm{x}\bm{y}+\sqrt{(r^2-\bm{x}^2)(r^2-\bm{y}^2)}}{r|\bm{y}-\bm{x}|}\big),\;\;\;&\alpha = n,
\end{cases}
\end{equation}
with 
\begin{equation*}
\rho(\bm{x},\bm{y}) = \frac{(r^2-|\bm{x}|^2)(r^2-|\bm{y}|^2)}{r^2|\bm{y}-\bm{x}|^2},\;\;\; \hat{C}_n^\alpha = \frac{\Gamma(n/2)}{2^\alpha\pi^{\frac{n}{2}}\Gamma^2(\alpha/2)}.
\end{equation*}

Now it remains to simulate the trajectory of $\{X^i_\ell\}_{\ell=1}^{L}$. To accomplish this, it is sufficient to determine the position of the next random point, denoted as $X^i_{\ell+1}$, based on the current location $X^i_{\ell}$.
Indeed, as long as we can determine the position of the current random point $X^i_{\ell+1}$ and its distance to the boundary $r_\ell$, we can utilize the Poisson kernel to determine the jump distance $J_\ell$, which is given by 
 \begin{eqnarray}
 \label{radius}
J_\ell:= J_\ell(\omega;r_\ell,\alpha)=  \frac{r_\ell}{\sqrt{B(1-\frac{\alpha}{2},\frac{\alpha}{2})- B^{-1}(\frac{\pi\,\omega}{\sin(\pi\alpha/2)};1-\frac{\alpha}{2},\frac{\alpha}{2})}},\;\;\;\omega\in(0,1),
 \end{eqnarray}
 where $B^{-1}(\cdot\,;a,b)$ denote the inverse function of  incomplete Beta function $B(\cdot\,;a,b)$.  
 With the jump distance $J_\ell$, we can obtain $X^i_{\ell+1}$ by multiplying it with the uniformly distributed spherical direction, that is,
 \begin{equation}
\label{XiforPoisson}
X^i_{\ell+1} = X^i_{\ell} + J_\ell\cdot
\begin{bmatrix}
\cos\theta_{1}\\
 \sin\theta_{1}\cos\theta_{2}\\\
 \cdots\cdots\\
\sin\theta_{1}\cdots\sin\theta_{d-2}\sin\theta_{d-1}
\end{bmatrix}.
\end{equation}
Note that we can generate another random variable $Y_{\ell+1}^i$ for inside the ball $\B^{n}_{r_{\ell}}$ centered at $X_\ell^i$ associated with the density function $\Q(\bx,\by)$. The method outlined above is applicable to the one-dimensional scenario as well, with the spherical directions replaced by left and right directions. However, for the sake of brevity and clarity, we will omit these straightforward details in our discussion.

\subsection{Generalized Jacobi functions and spectral approximation}
Let $\Omega = \Lambda:=(-1,1),$ the classical Jacobi polynomials are mutually orthogonal with respect to the weight function $\omega^{(\alpha,\beta)}(x)=(1-x)^\alpha(1+x)^\beta$ over the interval $\Omega:=(-1,1)$, i.e.
\begin{equation}\label{orthjacobi}
\int_{-1}^1 P_n^{(\alpha,\beta)}(x)\,
P_m^{(\alpha,\beta)}(x)\,
\omega^{(\alpha,\beta)}(x)\,\d x =\gamma^{(\alpha,\beta)}_n \delta_{nm},
\end{equation}
where $\delta_{nm}$ is the Kronecker delta and
\begin{equation}
\gamma^{(\alpha,\beta)}_n= \frac{2^{\alpha+\beta+1}\Gamma(n+\alpha+1)\Gamma(n+\beta+1)}{(2n+\alpha+\beta+1)n!\Gamma(n+\alpha+\beta+1)}.
\nonumber
\end{equation}

For $\alpha,\beta>-1$, let $\{x_j^{(\alpha,\beta)}, \omega_{ j}^{(\alpha,\beta)}\}_{j=0}^{N_x}$ be the set of Jacobi-Gauss (JG) quadrature nodes and weights, which enjoys the exactness
\begin{equation}
\int_{-1}^1 \phi(x) \omega^{(\alpha,\beta)}(x) d x=\sum_{j=0}^{N_x} \phi\big(x_j^{(\alpha,\beta)}\big) \omega^{(\alpha,\beta)}_j, \quad \forall \phi \in \mathcal{P}_{2 {N_x}+1},
\nonumber
\end{equation}
where $\mathcal{P}_N$ is the set of all polynomials of degree at most $N$. Let $I^{(\alpha,\beta)}_{N_x} u$ be the Jacobi-Gauss Lagrange polynomial interpolation of $u\in C(\Lambda)$ defined by
\begin{equation}\label{ujex}
I^{(\alpha,\beta)}_{N_x}u(x)=\sum^{N_x}_{j=0}u(x^{(\alpha,\beta)}_j)h_j(x) \in \mathcal{P}_{2 {N_x}+1},
\end{equation}
where the interpolating basis polynomials $\{h_j\}_{j=0}^{N_x}$ can be expressed by
\begin{equation*}
h_j(x)=\sum_{n=0}^{N_x} a_{n j} P_n^{(\alpha,\beta)}(x), \;\;\text{with}\;\,a_{n j}:=\frac{\omega_j}{\gamma_n^{(\alpha,\beta)}} P_n^{(\alpha,\beta)}(x_j). 
\end{equation*}

We infer from \cite[(7.12)]{Grubb2015} that the solution of \eqref{uf} is singular near the boundary $\partial\Omega$ which behaves like ${\rm dist}(x,\partial\Omega)^{\frac{\alpha}2}v(x)$, where ${\rm dist}(x,\partial\Omega)$ denotes the distance from $x\in\Omega$ to $\partial\Omega$ and $v$ is a smooth function.
To match the singularity, we resort to the generalized Jacobi functions (cf. \cite{Chen2016,Chen2018}) with index $(\frac{\alpha}{2},\frac{\alpha}{2})$ as the basis function, that is,
\begin{equation}\label{gJf11}
\Ja_n(x)=(1-x^2)^{\frac{\alpha}{2}}P_n^{(\frac{\alpha}{2},\frac{\alpha}{2})}(x)=
\oa(x)\, P_n^{(\frac{\alpha}{2},\frac{\alpha}{2})}(x),
\quad x\in\Lambda,
\end{equation}
which together with \eqref{orthjacobi} leads to
\begin{equation}\label{orthGJF}
\int_{-1}^1 \Ja_n(x)\Ja_m(x)\,
\oaa(x)\,\d x =\gamma^{(\frac{\alpha}{2},\frac{\alpha}{2})}_n \delta_{nm}.
\end{equation}
In the above, the notations $\omega^{\frac{\alpha}{2}}(x)$ and $\omega^{-\frac{\alpha}{2}}(x)$ are short for $\omega^{(\frac{\alpha}{2},\frac{\alpha}{2})}(x)$ and $\omega^{(-\frac{\alpha}{2},-\frac{\alpha}{2})}(x)$, respectively.
The GJFs apparently cover the case of polynomials $ \mathcal{J}_n^{1}(x)=(1-x^2)P^{(1,1)}_n(x)=\frac{2(n+1)}{2n+3}(L_n(x)-L_{n+2}(x))$ as $\alpha\rightarrow2$. 
Throughout this paper, we set $\alpha\in (0,2]$, and then we claim that our method is unified for classical and fractional PDEs.
We also define the finite-dimensional fractional-polynomial space:
\begin{equation*}
\fna=\big\{(1-x^2)^{\frac{\alpha}{2}}\psi:\psi\in \mathcal{P}_{N_x}\big\}={\rm span}\big\{\Ja_n(x):0\leq n\leq N_x\big\}.
\end{equation*}

To simplify the notation, we write $x_j:=x_j^{(\frac{\alpha}{2},\frac{\alpha}{2})}$ and $\omega_j:=\omega_j^{(\frac{\alpha}{2},\frac{\alpha}{2})}$ the Jacobi-Gauss points and weights, respectively.  We then define the following generalized Lagrange interpolation
\begin{equation}\label{unx}
\ina u(x)=\sum^{N_x}_{j=0}u(x_j)\lja\in\fna,
\end{equation}
where the generalized Lagrange interpolating basis is defined by
\begin{equation}\label{fracLagbasxx}
\lja:=\Big(\frac{1-x^2}{1-x_j^2}\Big)^{\frac{\alpha}2}h_j(x)=\Big(\frac{1-x^2}{1-x_j^2}\Big)^{\frac{\alpha}2}\prod_{\substack{0 \leq m \leq {N_x} \\ m \neq j}} \frac{x-x_m}{x_j-x_m}.
\end{equation}
One verifies readily that $l^{\frac{\alpha}{2}}_j(x_i)=\delta_{ij}$ for $0\leq i,j\leq N.$
Clearly, $\lja$ is a non-polynomials basis, we find their relation with generalized Jacobi functions, that is,
\begin{equation}\label{interpolationex}
\lja=\sum_{n=0}^{N_x} c_{nj}\Ja_n(x),\;\;\text{where}\;\,c_{nj}=\frac{1}{\gamma_n^{(\frac{\alpha}{2},\frac{\alpha}{2})}} (1-x_j^2)^{-\frac{\alpha}2}  P_n^{(\frac{\alpha}{2},\frac{\alpha}{2})}(x_j)\omega_j.
\end{equation}

According to \cite[Theorem 2]{Mao2016}, we have
\begin{equation} \label{derpol}
(-\Delta)^{\frac{\alpha}{2}} \Ja_n(x)= \frac{\Gamma(n+\alpha+1)}{n!} P^{(\frac{\alpha}{2},\frac{\alpha}{2})}_n(x),
\end{equation}
which together with \eqref{unx} and \eqref{interpolationex} leads to
\begin{equation}\label{fracujex}
\begin{split}
(-\Delta)^{\frac{\alpha}{2}}  \ina u(x)&=\sum_{j=0}^{N_x} u(x_j)\Big(\sum_{n=0}^{N_x}c_{nj}(-\Delta)^{\frac{\alpha}{2}} \Ja_n(x)\Big)\\&=\sum_{j=0}^{N_x} u(x_j)\Big(\sum_{n=0}^{N_x} c_{nj} \frac{\Gamma(n+\alpha+1)}{n!} P_n^{(\frac{\alpha}{2},\frac{\alpha}{2})}(x)\Big)
\\&=\sum_{n=0}^{N_x} \tilde{u}_n P_n^{(\frac{\alpha}{2},\frac{\alpha}{2})}(x),
\;\;\text{with}\;\,\tilde{u}_n=\sum_{j=0}^{N_x} u(x_j)c_{nj}\frac{\Gamma(n+\alpha+1)}{n!}.
\end{split}
\end{equation}

\subsection{Implementation of spectral Monte Carlo iteration algorithm (SMC)}\label{algorm}
In this section, we restrict our attention to the following fractional Poisson equation with $\alpha\in(0,2]$ and homogeneous Dirichlet conditions in one dimension:
\begin{equation}\label{uf0}
\begin{cases}
(-\Delta)^{\frac{\alpha}2}u(x)=f(x),\;&x\in \ \Omega,\\[4pt]
u(x)= 0, &x\in \ \Omega^c .
\end{cases}
\end{equation}
where the limiting case  $\alpha\rightarrow2$ becomes the classical Poisson equation reads
\begin{equation}\label{uf2} 
\begin{cases}
-\Delta u(x)=f(x),\;&x \in \ \Omega,\\[4pt]
u(x)= 0, &x \in \ \partial\Omega .
\end{cases}
\end{equation}
Note that for the Poisson equation with nonhomogeneous Dirichlet boundary conditions \eqref{uf0}, we set $u(x)$ to define on $\mathbb R$ and denote its restriction on the finite domain $\Omega$ by $u_\Omega(x)=u(x)|_{x\in \Omega}.$ Given $g(x)$ defined on $\Omega^c,$ we look for
\begin{equation*}\label{uBC}
u(x)=\begin{cases}
u_\Omega(x),\quad & x\in \Omega,\\
g( x), &  x\in {\Omega}^c,
\end{cases}
\end{equation*}
 i.e., the unknown $u_\Omega(x)$, such that $(-\Delta)^{\frac{\alpha}2}u(x)=f( x)$ in $\Omega.$ It is clear that we can write the solution as
$
u(x)= \tilde u_\Omega(x)+ \tilde g(x), \;x\in {\mathbb R},
$
where $\tilde u_\Omega$ (resp. $\tilde g$) is the zero extension of  $u_\Omega$ (resp. $g$)  on $\Omega$ (resp. $\Omega^c$) to $\mathbb R.$
 Then the  problem of interest becomes $(-\Delta)^{\frac{\alpha}2}\tilde u_\Omega(x)=f(x)- (-\Delta)^{\frac{\alpha}2}\tilde g(x).$ While for $\alpha=2$, for the nonhomogeneous boundary conditions, we can direct split $u(x)= \tilde u(x)+ \tilde g(x),$ with $\tilde u|_{\partial\Omega}=0$, then solve $\tilde u$ by \eqref{uf2}.

The traditional Monte Carlo method often faces limitations in achieving high accuracy. To address this, various variance reduction techniques combined with spectral methods have been developed \cite{Gobet2004,Gobet2005}. More precisely, we obtain an initial value using the traditional Monte Carlo method and construct a residual equation based on this initial value. By iteratively solving the residual equation using Monte Carlo simulations and updating the initial value, we can progressively improve the accuracy of our computations. 
Now, we describe the numerical implementation of the highly accurate stochastic iteration algorithm for \eqref{uf0}. To summarize, the detailed implementation involves the following steps:

\smallskip
\underline{\bf Step 1:}\, Using the walk-on-spheres method to compute the initial iteration solution of \eqref{uf0} denoted by $u^{(1)}_\ast(x)$ at Jacobi-Gauss points $\{x_j:=x^{(\frac{\alpha}2,\frac{\alpha}2)}_j\}_{j=0}^{N_x}$:
\begin{equation}
u^{(1)}_\ast(x_j) =  \frac{1}{M}\sum_{i=1}^M\Big[\sum_{\ell=0}^{L- 1} \zeta(X_{j,\ell}^i) \mathbb{E}^\ast_{\Q}\big[f(Y_{j,\ell+1}^i)\big]\Big],\;\;\;0\leq j\leq N_x,
\end{equation}
where $\mathbb{E}^\ast_{\widetilde{Q}}$ is defined in \eqref{zetaE}, $M$ denotes the number of paths and the initial position of random process $X^{i}_{j,0}=x_j$, with $1\leq i\leq M$.

\smallskip
\underline{\bf Step 2:}\, Compute the fractional Laplacian of $u^{(1)}_\ast(x)$ from initial data $\{u^{(1)}_\ast(x_j)\}_{j=0}^{N_x}$, i.e., $ (-\Delta)^{\frac{\alpha}2}u^{(1)}_\ast(x)$. 
More precisely, we employ the generalized Lagrange interpolation \eqref{unx} and write $u^{(1)}_\ast(x)=\sum^{N_x}_{j=0}u^{(1)}_\ast(x_j)l^{(\frac{\alpha}{2},\frac{\alpha}{2})}_j(x),$ then by \eqref{fracujex} the exact fractional Laplacian of $u^{(1)}_\ast$ reads
\begin{equation*} 
\begin{split}
(-\Delta)^{\frac{\alpha}{2}} u^{(1)}_\ast(x) =\sum_{n=0}^{N_x} \tilde{u}_n P_n^{(\frac{\alpha}{2},\frac{\alpha}{2})}(x),\;\;\text{where}\;\,\tilde{u}_n=\sum_{j=0}^{N_x} u^{(1)}_\ast(x_j)c_{nj}\frac{\Gamma(n+\alpha+1)}{n!}.
\end{split}
\end{equation*}

\vspace{-6pt}
\underline{\bf Step 3:}\, Using the walk-on-spheres method to compute the residual function $\varepsilon=u-u^{(1)}_\ast$ from the following 
\begin{equation}\label{cvm}
\begin{cases}
(-\Delta)^{\frac{\alpha}2}\varepsilon= f^{(1)}, & \text { in } \Omega,\\
\varepsilon = 0, & \text { on } \Omega^c,
\end{cases}
\end{equation}
where the updated right hand side source function $f^{(1)} = f- (-\Delta)^{\frac{\alpha}2} u^{(1)}_\ast.$  Similarly, through the Feynman-Kac formula and walk-on-spheres method, we have the numerical solution $\varepsilon^{(1)}_\ast$ of \eqref{cvm}
\begin{equation*}
\varepsilon^{(1)}_\ast(x_j) = \frac{1}{M}\sum_{i=1}^M \Big[\sum_{\ell=0}^{L- 1} \zeta (X^i_{j,\ell}) \mathbb{E}^\ast_{\widetilde{Q}}\big[f^{(1)}(Y^i_{j,\ell+1})\big]\Big],\;\;\;0\leq j\leq N_x.  
\end{equation*}

\underline{\bf Step 4:}\, Update the solution by $u^{(2)}_\ast(x)=u^{(1)}_\ast(x)+\varepsilon^{(1)}_\ast(x).$ Clearly, we can expect that the updated solution $u^{(2)}_\ast$ is more accurate than the previous stage $u^{(1)}_\ast$.

In actual computation,  we repeat Steps 2-4 until the maximum absolute difference between the two
solutions $u^{(k-1)}_\ast$ and $u^{(k)}_\ast$ is less than the desired tolerance or the machine accuracy. 

\section{The error estimates for SMC}\label{sect3} 
\setcounter{lem}{0} \setcounter{thm}{0}  \setcounter{rem}{0} 	
In this section, we shall analyze and characterize the convergence of the iteration Monte Carlo algorithm proposed in section \ref{algorm}, where we aim to estimate the error bounds in the sense of weak convergence 
\begin{equation}
E^\infty_k:=\max_{1\leq i \leq N_x}|\mathbb{E}(u^{(k)}_\ast(x_i)-u(x_i))|,
\end{equation}
where $\{x_i\}_{i=0}^{N_x}$ denote the Jacobi-Gauss points.

Then, we can derive the following error estimate for interpolation operator in \eqref{unx} under $L^\infty$-norm, which will be used later.

\begin{lemma}\label{lem3.1}
Let $\alpha\in (0,2]$ and denote $\tilde u(x):=\oaa(x)u(x)$. Suppose that $\tilde{u}^{(r-1)}$ is absolutely continuous on $\Lambda$ for some $r\geq 1$ and its $r$-th derivative $\partial_x^r\tilde{u}$ is of bounded variation ${\rm Var}[\partial_x^r\tilde{u}]<\infty$, then we have
\begin{equation}\label{estimate_interpolation}
\big\| \ina u -  u \big\|_{L^{\infty}}
\leq c N_x^{\frac{\alpha}{2}-\frac12-r},
\end{equation}
where $c$ is a positive constant independent of $r$ and $N_x$.
\end{lemma}
\begin{proof}
Recall the interpolation error $I^{(\alpha,\beta)}_{N_x}$ (cf. \cite{Xiang2016}): if $f(x)$ has an absolutely continuous $(r-1)$-st
derivative $\partial_x^{r-1}f(x)$ on $\Lambda$ for some $r\ge 1$ and its $r$-th derivative $\partial_x^{r}f(x)$ is of bounded variation 
${\rm Var}[\partial_x^{r}f]$, then
\begin{equation}\label{finftynorm}
\|f-I^{(\alpha,\beta)}_{N_x}f\|_{L^\infty}\leq c N^{\gamma-{\frac{1}{2}-r}}_x,\;\;\;\text{with}\;\gamma=\max\{\alpha,\beta\}.
\end{equation}
In view of \eqref{ujex} and \eqref{unx}, we can easily check that 
\begin{equation}\begin{split}\label{identityinterpolation}
\ina u &=\sum^{N_x}_{j=0}u(x_j)\lja=\sum^{N_x}_{j=0}u(x_j)\Big(\frac{1-x^2}{1-x_j^2}\Big)^{\frac{\alpha}2}h_j(x)
\\&=\oa(x)\sum^{N_x}_{j=0}\frac{u(x_j)}{(1-x_j^2)^{\frac{\alpha}{2}}}h_j(x)=\oa(x) \Big(I_{N_x}^{(\frac{\alpha}{2},\frac{\alpha}{2})}\big(\oaa(x)u(x)\big)\Big).
\end{split}\end{equation}
Then, we derive from \eqref{finftynorm} and \eqref{identityinterpolation} that
\begin{equation*}
\begin{split}
\big\| \ina u -  u \big\|_{L^{\infty}}
&=\big\|\oa (I_{N_x}^{(\frac{\alpha}{2},\frac{\alpha}{2})} (\oaa u) - (\oaa u)) \big\|_{L^{\infty}}
\\&\leq \|\oa \|_{L^\infty}\big\|I_{N_x}^{(\frac{\alpha}{2},\frac{\alpha}{2})} (\oaa u) - (\oaa u)\big\|_{L^{\infty}}
\\&\leq c\big\|I_{N_x}^{(\frac{\alpha}{2},\frac{\alpha}{2})} (\oaa u) - (\oaa u) \big\|_{L^{\infty}}\leq cN^{\frac{\alpha}{2}-\frac12-r}_x.
\end{split}
\end{equation*}
This ends the proof.
\end{proof}

To perform the error bound of $u(x)-u_\ast(x)$,  we first introduce an auxiliary function
\begin{equation}\label{phi_appr2}
\hphi_{\bx}(f,g,X_t) =  g(X_{L})+\sum_{\ell=0}^{L-1}\zeta(X_\ell)\mathbb{E}_{\widetilde{Q}}[f(Y_{\ell+1})],
\end{equation}
which replaces $\mathbb{E}^\ast_{\widetilde{\mathbb{Q}}}$ with $\mathbb{E}_{\widetilde{\mathbb{Q}}}$ in $\tphi_{x}(\cdot)$ (cf.\eqref{iresolu}). 
In view of \eqref{iresolu} and \eqref{phi_appr2}, we define the error function 
\begin{equation}\label{e_def}
e(f,M_1,x) =\tphi_{\bx}(f,g,X_t) - \widehat{\Phi}_{\bx}(f,g,X_t) = \sum_{\ell=0}^{L-1}\zeta(X_\ell)\Big(\mathbb{E}^\ast_{\widetilde{\mathbb{Q}}}[f(Y_{\ell+1})]-\mathbb{E}_{\widetilde{\mathbb{Q}}}[f(Y_{\ell+1})]\Big).
\end{equation}
By virtue of the SMC in Section \ref{algorm}, we can rewrite the $u^{(k)}_\ast(x)$ in terms of $u^{(k-1)}_\ast(x)$ as 
\begin{equation}\label{iteration_formula}
u^{(k)}_\ast(x_j) = u^{(k-1)}_\ast(x_j) + \frac{1}{M}\sum_{i=1}^{M}\Phi^{\ast}_{x_j}\big((-\Delta)^{\frac{\alpha}{2}}(u-u^{(k-1)}_\ast),(u-u^{(k-1)}_\ast)|_{\Omega^c},X_{j,t}^i\big),
\end{equation}
where the initial position of random process $X_{j,t}^i$ is $X^{i}_{j,0}=x_j$, with $1\leq i\leq M$ and $0\leq j\leq N_x.$
For simplicity, we define the following energy norm
\begin{equation*}
\|u\|_{H^\alpha(\Lambda)}:=\Big(\int_{\Lambda}\big((-\Delta)^{\frac{\alpha}{2}}u(x)\big)^2\,{\rm d}\by\Big)^{\frac{1}{2}}.
\end{equation*}
We now present the convergence result for the bias of the iterative algorithm. 
\setcounter{theorem}{0}
\begin{theorem}\label{thm3.3}
Let $\alpha \in (0,2)$, $u(x)$ be the solution of \eqref{uf} and $u^{(k)}_\ast(x)$ be the numerical solution of SMC described in section {\rm\ref{algorm}}, and denote $\tilde u(x):=\oaa(x)u(x)$. Suppose that $\partial_x^{r-1}\tilde{u}$ is absolutely continuous on $\Lambda$ for some $r\geq 1$ and its $r$-th derivative $\partial_x^r\tilde{u}$ is of bounded variation ${\rm Var}(\partial_x^r\tilde{u})<\infty$. Then, we have  
\begin{equation}\label{err_e_0^k} 
E_k^{\infty} \leq \rho E_{k-1}^{\infty} + cN_x^{\frac{\alpha}{2}-\frac12-r} +c M_1^{-\frac{1}{2}}\Vert u-\mathcal{I}_{N_x}^\frac{\alpha}{2}u\Vert_{H^\alpha },
\end{equation}
where the constant $\rho = cN_xM_1^{-\frac{1}{2}}\max_{0\leq j \leq N_x}\big\|(-\Delta)^{\frac{\alpha}{2}}l_j^\frac{\alpha}{2}\big\|_{L^2}$, and the constant $c$ is independent of parameters $k,N_x,M_1$. 
For any fixed $N_x$ and sufficiently large $M_1$, we have $\rho < 1$. Consequently, the sequence $\{E^\infty_k\}_k$ converges geometrically at rate $\rho$, up to a threshold equal to
\begin{equation}\label{err_limsup1} 
E_k^{\infty} \leq \frac{1}{1-\rho}\big[cN_x^{\frac{\alpha}{2}-\frac12-r} + c M_1^{-\frac{1}{2}}\Vert u-\mathcal{I}_{N_x}^\frac{\alpha}{2}u\Vert_{H^\alpha}\big].
\end{equation}
\end{theorem}
\begin{proof}
 Clearly, \eqref{err_limsup1} is a direct consequence of \eqref{err_e_0^k}, so we will only prove \eqref{err_e_0^k}. 
Since $u^{(k)}_\ast(x)$ is obtained by interpolating at the Jacobi-Gauss points $\{x_j\}_{0\leq j \leq N_x}$, we have $\mathcal{I}_{N_x}^\frac{\alpha}{2}u^{(k-1)}_\ast(x)$ $=u^{(k-1)}_\ast(x)$, then we get that
\begin{equation}\label{decomposition}
\begin{aligned}
u(x)- u^{(k-1)}_\ast(x) &= \big(u(x)-\mathcal{I}_{N_x}^\frac{\alpha}{2}u(x)\big)+\big(\mathcal{I}_{N_x}^\frac{\alpha}{2}u(x) - u^{(k-1)}_\ast(x)\big)\\
&= \big(u(x)-\mathcal{I}_{N_x}^\frac{\alpha}{2}u(x)\big)+\mathcal{I}_{N_x}^\frac{\alpha}{2}\big(u(x) - u^{(k-1)}_\ast(x)\big).
\end{aligned}
\end{equation}
To distinguish between the simulations conducted before stage $k$ and those at stage $k$, we define the following conditional expectations
\begin{equation*}
\mathbb{E}^{(k-1)}[Y] = \mathbb{E}\big[Y|\sigma_1,\sigma_2,\cdots,\sigma_{k-1}\big],
\end{equation*}
where $\sigma_{l}$ represents the information for the $l$-th iteration, with $l=1,2,\cdots,k-1$. 
Recall the Law of total expectation, we have 
\begin{equation}\label{eyy}
\mathbb{E}[Y] = \mathbb{E}\big[\mathbb{E}^{(k-1)}[Y]\big],\;\;\; k\ge2.
\end{equation}
Thus, we first consider the conditional expectation $\mathbb{E}^{(k-1)}[\cdot]$ of the above identity.
Substituting it into right-hand side of \eqref{iteration_formula} and taking the conditional expectation on both sides leads to
\begin{equation}\label{u^{(k)}(x_i)}
\begin{split}
\mathbb{E}^{(k-1)}[u^{(k)}_\ast(x_j)]  &= \mathbb{E}^{(k-1)}[u^{(k-1)}_\ast(x_j)] + \mathbb{E}^{(k-1)}[B_1] + \mathbb{E}^{(k-1)}[B_2]
\\&=u^{(k-1)}_\ast(x_j) + \mathbb{E}^{(k-1)}[B_1] + \mathbb{E}^{(k-1)}[B_2],
\end{split}
\end{equation} 
where 
\begin{equation*}
 \begin{split}
&B_1 = \frac{1}{M}\sum_{i=1}^{M}\tphi_{x_j}\big((-\Delta)^{\frac{\alpha}{2}}(u-\mathcal{I}_{N_x}^\frac{\alpha}{2}u),(u-\mathcal{I}_{N_x}^\frac{\alpha}{2}u)|_{\Omega^c},X^i_{j,t}\big),\\&
B_2 = \frac{1}{M}\sum_{i=1}^{M}\tphi_{x_j}\big((-\Delta)^{\frac{\alpha}{2}}(\mathcal{I}_{N_x}^\frac{\alpha}{2}u-u^{(k-1)}_\ast),(\mathcal{I}_{N_x}^\frac{\alpha}{2}u-u^{(k-1)}_\ast)|_{\Omega^c},X^i_{j,t}\big).
\end{split}
\end{equation*}
We first compute $\mathbb{E}^{(k-1)}[B_1]$. Since $B_1$ is independent of the previous stages $\{u_\ast^{(j)}\}_{j=0}^{k-1}$, implying $\mathbb{E}^{(k-1)}[B_1] =\mathbb{E}[B_1]$, we then deduce from the above equations that
\begin{eqnarray*}
&&\mathbb{E}^{(k-1)}[B_1] =\mathbb{E}[B_1]=\mathbb{E}\Big[ \frac{1}{M}\sum_{i=1}^{M}\tphi_{x_j}\big((-\Delta)^{\frac{\alpha}{2}}(u-\mathcal{I}_{N_x}^\frac{\alpha}{2}u),(u-\mathcal{I}_{N_x}^\frac{\alpha}{2}u)|_{\Omega^c},X^i_{j,t}\big)\Big]
\\&&\hspace{49pt}= \frac{1}{M}\sum_{i=1}^{M}\mathbb{E}\Big[\tphi_{x_j}\big((-\Delta)^{\frac{\alpha}{2}}(u-\mathcal{I}_{N_x}^\frac{\alpha}{2}u),(u-\mathcal{I}_{N_x}^\frac{\alpha}{2}u)|_{\Omega^c},X^i_{j,t}\big)\Big]
\\&&\hspace{49pt}= \mathbb{E}\Big[\tphi_{x_j}\big((-\Delta)^{\frac{\alpha}{2}}(u-\mathcal{I}_{N_x}^\frac{\alpha}{2}u),(u-\mathcal{I}_{N_x}^\frac{\alpha}{2}u)|_{\Omega^c},X_{j,t}\big)\Big]
\\&&\hspace{49pt}=\mathbb{E}\Big[\hphi_{x_j}\big((-\Delta)^{\frac{\alpha}{2}}(u-\mathcal{I}_{N_x}^\frac{\alpha}{2}u),(u-\mathcal{I}_{N_x}^\frac{\alpha}{2}u)|_{\Omega^c},X_{j,t}\big)\Big]\\
&&\hspace{49pt}\quad+ \mathbb{E}\Big[(\tphi_{x_j}-\hphi_{x_j})\big((-\Delta)^{\frac{\alpha}{2}}(u-\mathcal{I}_{N_x}^\frac{\alpha}{2}u),(u-\mathcal{I}_{N_x}^\frac{\alpha}{2}u)|_{\Omega^c},X_{j,t}\big)\Big]\\
&&\hspace{49pt}=\mathbb{E}\Big[\hphi_{x_j}\big((-\Delta)^{\frac{\alpha}{2}}(u-\mathcal{I}_{N_x}^\frac{\alpha}{2}u),(u-\mathcal{I}_{N_x}^\frac{\alpha}{2}u)|_{\Omega^c},X_t\big)\Big]+ \mathbb{E}\Big[e\big((-\Delta)^{\frac{\alpha}{2}}(u-\mathcal{I}_{N_x}^\frac{\alpha}{2}u),M_1,x_j\big)\Big],
\end{eqnarray*}
which, together with the fact that $\mathbb{E}[\hphi_{x_j}(\cdot,\cdot,\cdot)]=\mathbb{E}[\Phi_{x_j}(\cdot,\cdot,\cdot)]$, implies 
\begin{equation}\label{EJ_1_1}
\mathbb{E}[B_1] = [u - \mathcal{I}_{N_x}^\frac{\alpha}{2}u](x_j) +\mathbb{E}\Big[e\big((-\Delta)^{\frac{\alpha}{2}}(u-\mathcal{I}_{N_x}^\frac{\alpha}{2}u),M_1,x_j\big)\Big].
\end{equation}
Then, we derive from  \eqref{e_def} and \eqref{u^{(k)}(x_i)} that
\begin{equation}\label{EB1}
\begin{split}
\Big|\mathbb{E}\Big[e\big((-\Delta)^{\frac{\alpha}{2}}(u-\mathcal{I}_{N_x}^\frac{\alpha}{2}u),M_1,x_j\big)\Big]\Big|& \leq c \mathbb{E}[L]\Big|(\E^\ast_{\Q}-\E_{\Q})\Big[(-\Delta)^{\frac{\alpha}{2}}(u-\mathcal{I}_{N_x}^\frac{\alpha}{2}u)\Big]\Big|\\
&\leq c M_1^{-\frac{1}{2}}\V^{\frac{1}{2}}\Big[(-\Delta)^{\frac{\alpha}{2}}(u-\mathcal{I}_{N_x}^\frac{\alpha}{2}u)\Big] \\
&\leq cM_1^{-\frac{1}{2}}\Big(\int_{\mathbb{B}_r}\Big((-\Delta)^{\frac{\alpha}{2}}(u-\mathcal{I}_{N_x}^\frac{\alpha}{2}u)\Big)^2\Q(\bx,\by)\,{\rm d}\by\Big)^{\frac{1}{2}}
\\& \leq cM_1^{-\frac{1}{2}}\Big(\int_{\mathbb{B}_r}\Big((-\Delta)^{\frac{\alpha}{2}}(u-\mathcal{I}_{N_x}^\frac{\alpha}{2}u)\Big)^2\,{\rm d}\by\Big)^{\frac{1}{2}}.
\end{split}
\end{equation}
Taking the absolute value of both sides of \eqref{EJ_1_1}, and combining with  \eqref{EB1}, and \eqref{estimate_interpolation}, we arrive at
\begin{equation}\label{Eb1bound} 
\big|\mathbb{E}[B_1]\big| \leq cN_x^{\frac{\alpha}{2}-\frac12-r} + cM_1^{-\frac{1}{2}}\Vert u-\mathcal{I}_{N_x}^\frac{\alpha}{2}u\Vert_{H^\alpha}.
\end{equation}
Next we estimate $\mathbb{E}^{(k-1)}[B_2]$. It is evident that 
\begin{equation}\label{u_uk_ident}
\mathcal{I}_{N_x}^\frac{\alpha}{2}u(x) - u^{(k-1)}_\ast(x) = \mathcal{I}_{N_x}^\frac{\alpha}{2}(u-u^{(k-1)}_\ast)(x) = \sum_{i=0}^{N_x}(u(x_i)-u^{(k-1)}_\ast(x_i))l_i^{\frac{\alpha}{2}}(x).
\end{equation}
Due to the mutual independence of the stochastic processes, we obtain from \eqref{u^{(k)}(x_i)} and \eqref{u_uk_ident} that
\begin{eqnarray*}
&&\mathbb{E}^{(k-1)}[B_2] = \mathbb{E}^{(k-1)}\Big[\tphi_{x_j}((-\Delta)^{\frac{\alpha}{2}}(\mathcal{I}_{N_x}^\frac{\alpha}{2}u-u^{(k-1)}_\ast),(\mathcal{I}_{N_x}^\frac{\alpha}{2}u-u^{(k-1)}_\ast)|_{\Omega^c},X_{j,t})\Big]\\
&&\hspace{49pt}= \sum_{i=0}^{N_x}(u(x_i)-u^{(k-1)}_\ast(x_i))\mathbb{E}^{(k-1)}\Big[\tphi_{x_j}((-\Delta)^{\frac{\alpha}{2}}l_i^{\frac{\alpha}{2}}(x),l_i^{\frac{\alpha}{2}}(x)\big|_{\Omega^c},X_{j,t})\Big]
\\&&\hspace{49pt}= \sum_{i=0}^{N_x}(u(x_i)-u^{(k-1)}_\ast(x_i))\mathbb{E}^{(k-1)}\Big[\hphi_{x_j}((-\Delta)^{\frac{\alpha}{2}}l_i^{\frac{\alpha}{2}}(x),l_i^{\frac{\alpha}{2}}(x)\big|_{\Omega^c},X_{j,t})\Big]\\
&&\hspace{49pt}\quad +\sum_{i=0}^{N_x}(u(x_i)-u^{(k-1)}_\ast(x_i))\mathbb{E}^{(k-1)}\Big[(\tphi_{x_j}-\hphi_{x_j})((-\Delta)^{\frac{\alpha}{2}}l_i^{\frac{\alpha}{2}}(x),l_i^{\frac{\alpha}{2}}(x)\big|_{\Omega^c},X_{j,t})\Big]\\
&&\hspace{49pt}= \sum_{i=0}^{N_x}(u(x_i)-u^{(k-1)}_\ast(x_i))\mathbb{E}^{(k-1)}\Big[\hphi_{x_j}((-\Delta)^{\frac{\alpha}{2}}l_i^{\frac{\alpha}{2}}(x),l_i^{\frac{\alpha}{2}}(x)\big|_{\Omega^c},X_{j,t})\Big]\\
&&\hspace{49pt}\quad +\sum_{i=0}^{N_x}(u(x_i)-u^{(k-1)}_\ast(x_i))\mathbb{E}^{(k-1)}\big[e((-\Delta)^{\frac{\alpha}{2}}l_i^{\frac{\alpha}{2}}(x),M_1,x_j)\big].
\end{eqnarray*}
Using the fact that $\mathbb{E}^{(k-1)}[\hphi_{x_j}(\cdot,\cdot,\cdot)] = \mathbb{E}^{(k-1)}[\Phi_{x_j}(\cdot,\cdot,\cdot)]$ and $$\mathbb{E}^{(k-1)}[\Phi_{x_j}((-\Delta)^{\frac{\alpha}{2}}l_i^{\frac{\alpha}{2}}(x),l_i^{\frac{\alpha}{2}}(x)|_{\Omega^c}, X_{j,t})]= l_i^{\frac{\alpha}{2}}(x_j) = \delta_{ij},$$
 we obtain that
\begin{equation*}\label{Ek-1J2}
\begin{aligned}
\mathbb{E}^{(k-1)}[B_2] 
=u(x_j)-u^{(k-1)}_\ast(x_j)+\sum_{i=0}^{N_x}(u(x_i)-u^{(k-1)}_\ast(x_i))\mathbb{E}^{(k-1)}\Big[e((-\Delta)^{\frac{\alpha}{2}}l_i^{\frac{\alpha}{2}}(x),M_1,x_j)\Big].
\end{aligned}
\end{equation*}
A combination of the above equation and \eqref{u^{(k)}(x_i)} leads to
\begin{small}
\begin{equation}\label{err_E_k-1}
\begin{aligned}
\mathbb{E}^{(k-1)}[u^{(k)}_\ast(x_j)-u(x_j)]  &= \mathbb{E}[B_1]  + \sum_{i=0}^{N_x}(u(x_i)-u^{(k-1)}_\ast(x_i))\mathbb{E}^{(k-1)}\Big[e((-\Delta)^{\frac{\alpha}{2}}l_i^{\frac{\alpha}{2}}(x),M_1,x_j)\Big].
\end{aligned}
\end{equation}
\end{small}
Taking the absolute value of both sides of \eqref{err_E_k-1}, and using \eqref{Eb1bound} and the triangle inequality, we derive that
\begin{equation} 
\begin{split}
&\big|\mathbb{E}^{(k-1)}[u^{(k)}_\ast(x_j)-u(x_j)]\big| \leq  cN_x^{\frac{\alpha}{2}-\frac12-r} + cM_1^{-\frac{1}{2}}\Vert u-\mathcal{I}_{N_x}^\frac{\alpha}{2}u\Vert_{H^\alpha }
\\&\quad+cN_x\max_{0\leq i\leq N_x}\Big|(u(x_i)-u^{(k-1)}_\ast(x_i))\Big| \Big|\mathbb{E}^{(k-1)}\Big[e((-\Delta)^{\frac{\alpha}{2}}l_i^{\frac{\alpha}{2}}(x),M_1,x_j)\Big]\Big|
\\& \leq  cN_x^{\frac{\alpha}{2}-\frac12-r} + cM_1^{-\frac{1}{2}}\Vert u-\mathcal{I}_{N_x}^\frac{\alpha}{2}u\Vert_{H^\alpha}+cN_xM_1^{-\frac{1}{2}}\max_{0\leq i\leq N_x}\Big|(u(x_i)-u^{(k-1)}_\ast(x_i))\Big|.
\end{split}
\end{equation}
Taking the expectation and pointwise maximum norm on both sides, and using the law 
of total expectation \eqref{eyy} and the fact that $\big|\mathbb{E}[\mathbb{E}^{(k-1)}[\cdot]]\big| \leq \mathbb{E}\big|\mathbb{E}^{(k-1)}[\cdot]\big|$ leads to the desired result \eqref{err_e_0^k}. 
This completes the proof.
\end{proof}
\begin{remark}{\em 
It is worthwhile to point out that the weak convergence error bound in \eqref{err_e_0^k} is independent of the number of simulations path $M$.    }
\end{remark}

\section{Space-time spectral Monte Carlo method for parabolic equation}\label{sect4} \setcounter{lem}{0} \setcounter{thm}{0}  \setcounter{rem}{0} 	
The proposed spectral Monte Carlo method for Poisson equation could be straightforwardly extended to the solution of the parabolic equation. 
We consider the following parabolic equation driven by $\alpha$-stable Lévy process in a bounded domain $\Omega\subset \mathbb R$ :
\begin{equation}\label{ufg}
\begin{cases}
\partial_t u(\bx,t)+(-\Delta)^{\frac{\alpha}2}u(\bx,t)=f(\bx,t),\;&(\bx,t) \in \ \Omega \times (0,\infty) ,\\[4pt]
u(\bx,t)=g(\bx,t),\quad & (\bx,t) \in \  \Omega^c \times (0,\infty) ,\\[4pt]
u(\bx,0)=u_0(\bx),\quad &\bx \in \Omega.
\end{cases}
\end{equation}
Thanks to the Feynman-Kac formula of \eqref{ufg} (cf. \cite{Su2023}), we can express the probability solution as
\begin{equation}\label{fkluxt}
\begin{split}
u(\bx,t) &= 	\mathbb{E}_{X_{0}^{\alpha }=\bm{x}}\Big[ u_0(X_t^{\alpha})\mathbb{I}_{\tau_{\Omega}^{\bx}>t}+g(X_{\tau_{\Omega}^{\bx}}^{\alpha},\tau_{\Omega}^{\bx})\mathbb{I}_{\tau_{\Omega}^{\bx}\leq t} \Big] + \mathbb{E}_{X_{0}^{\alpha }=\bm{x}}\Big[\int_{0}^{t \wedge \tau_{\Omega}^{\bx}}f(X^{\alpha}_{s},s)\,{\rm d}s\Big]
\\&:=\mathbb{E}\big[\Psi_{\bx,t}(f,g,u_0,X^\alpha_t)\big],
\end{split}
\end{equation}
where $\tau_{\Omega}^{\bx}$ is the first exit time to domain $\Omega$ and $X^{\alpha}_{\tau_{\Omega}^{\bx}}$ is the first exit location on the domain $\mathbb R^{d}\setminus\Omega$, and $\mathbb{E}_{\bm{x}}[\tau_{\Omega}^{\bx}]<\infty$, for all $\bx\in\Omega$.

\subsection{Stochastic algorithm for parabolic equation}
In this part, we will provide a detailed implementation of a stochastic method for computing the approximate solution $u(\bx_j, t_n)$ at given points $\{\bx_j\}$ and $\{t_n\}$. Note that the algorithm description provided in the part only focuses on the high-dimensional case, while the one-dimensional case can be easily recognized as a special case.
It is evident that, for a given time subdivision, there exists a correlation
between the expectation of the movement time of the stochastic process $X_t^\alpha$ within the ball and the length of the time interval, which plays a crucial role in the construction of our algorithm. Thus, for any given grid points $t_n$, we divide the time interval $[0,t_n]$ into $N$ equal parts:
\begin{equation}\label{time division}
\mathcal{T}:=\big\{0=t_{n,0}<t_{n,1}<t_{n,2}<\cdots<t_{n,N-1}<t_{n,N} = t_n\big\},
\end{equation}
where $\Delta t_n = t_{n,\ell}-t_{n,\ell-1} = \frac{t_n}{N}$, with $1\leq \ell\leq N$. To this ends, we define the exit time of the process $X_t^\alpha$ from $\mathbb{B}_r^d(X_{t_{n,\ell-1}}^\alpha)$ as
\begin{equation}
\tau_\ell:= \inf \{s:X_s^\alpha \notin \mathbb{B}_r^n(X_{t_{n,\ell-1}}^\alpha)\},\quad 1\leq \ell\leq N.
\end{equation}

In each time interval $[t_{n,\ell-1},t_{n,\ell})$, the process $X_t^\alpha$ starts from the location $X_{\tau_{n,\ell}}^\alpha$. Within this time interval, the trajectory of $X_t^\alpha$ remains within a ball $\mathbb{B}_r^d(X_{t_{n,\ell-1}}^\alpha)$, and the stochastic process $X_t^\alpha$ takes a jump at time $t_{n,\ell}$ and leaves ball $\mathbb{B}_r^d(X_{t_{n,\ell-1}}^\alpha)$. Hence, we can directly set $\tau_\ell\approx \Delta t_n$. 
The radius of the ball can be chosen  as $r=(\Delta t_n/\widetilde{C}_{d}^\alpha)^{1/\alpha}$ with constant $\widetilde{C}_{d}^\alpha = \frac{\Gamma(\frac d2)}{2^{\alpha}\Gamma(1 + \frac \alpha 2)\Gamma(\frac{d+\alpha}{2})}.$
In this setting, the motion path of the process $X_t^\alpha$ can be approximated by the motion paths of the balls, and each ball, with a fixed radius, is not required to be tangent to the boundary, which differs slightly from the inscribed balls in \cite{Sheng2023}. Now we are ready to present the following result on the location of $X^{\alpha}_{t_{n,\ell}}$ from $X^{\alpha}_{t_{n,\ell-1}}$, which states (cf.\cite{Sheng2024})
\begin{lemma}\label{Xlocation}
Let $\alpha\in(0,2]$ and $X^{\alpha}_{t_{n,\ell-1}}$ be the location of $\alpha$-stable Lévy process at $t_{n,\ell-1}$,
then the location of $\alpha$-stable Lévy process $X^{\alpha}_{t_{n,\ell}}$ can be computed by
\begin{equation}
\label{Xi}
X^{\alpha}_{t_{n,\ell}} = X^{\alpha}_{t_{n,\ell-1}} + J_\ell\cdot
\begin{bmatrix}
\cos\theta_{1}\\
 \sin\theta_{1}\cos\theta_{2}\\\
 \cdots\cdots\\
\sin\theta_{1}\cdots\sin\theta_{d-2}\sin\theta_{d-1}
\end{bmatrix},
\end{equation}
where the $\ell$-th jump distance is given by
 \begin{equation}\label{rhok}
J_\ell:=J_\ell(\omega;r,\alpha)=  \frac{r}{\sqrt{B(1-\frac{\alpha}{2},\frac{\alpha}{2})- B^{-1}(\frac{\pi\,\omega}{\sin(\pi\alpha/2)};1-\frac{\alpha}{2},\frac{\alpha}{2})}},\;\;\;\omega\in(0,1).
 \end{equation}
\end{lemma}

Different from the simulation for Poisson equation in \eqref{XiforPoisson}, here the radius of each ball is fixed, which is determined by the time interval $\Delta t_n$.
For any given $\bm{x}_j \in \Omega$ and $t_n$, we draw a ball $\mathbb{B}_r^d(\bm{x}_j)$ inside the domain $\Omega$. Then, using Lemma\,\ref{Xlocation}, we can easily obtain the location of a sequence of balls with radius $r$ centered at
\begin{equation*}
\bx_j=X^{\alpha}_{t_{n,0}}<X^{\alpha}_{t_{n,1}}<\cdots<X^{\alpha}_{t_{n,N}}.
\end{equation*}
Now, we only need to determine whether the location of the center of $\ell$-th ball $X^\alpha_{n,\ell}$ with $1\leq\ell\leq N$ is within the domain $\Omega$. For this purpose, we define the following index set
\begin{equation}\label{mast}
 L=\begin{cases}
     N,\;\; &\text{if}\; X^\alpha_{n,\ell}\in \Omega,\;\forall 1\leq \ell\leq N,\\[2pt]
     p,\;\; &\text{if}\; X^\alpha_{n,\ell}\in \Omega,\;\forall 1\leq \ell\leq p<N\; \&\;X^\alpha_{n,p+1}\notin \Omega.
 \end{cases}  
\end{equation}
Therefore, we can obtain the $i$-th experiment of stochastic solution at $\bx_j$ and $t_n$: 
\begin{equation}\label{phidef1}
\Psi^\ast_{\bm{x}_j,t_n}(f,g,u_0,X_\ell^i) = u_0(X^{i}_{t_{n,L}})+\mathcal{Q}^i_{L}[f],\;\;\; \text{if}\; X^i_{n,\ell}\in \Omega,\;\forall 1\leq \ell\leq N,
\end{equation}
and 
\begin{equation}\label{phidef2}
\Psi^\ast_{\bm{x}_j,t_n}(f,g,u_0,X_\ell^i) = g(X^{i}_{t_{n,L}},t_{L})+\mathcal{Q}^i_{L}[f],\;\;\; \text{if}\; X^\alpha_{n,\ell}\in \Omega,\;\forall 1\leq \ell\leq p<N\; \&\;X^\alpha_{n,p+1}\notin \Omega,
\end{equation}
where $X^{i}_{t_{n,\ell}}$ is short for $X^{\alpha,i}_{t_{n,\ell}}$, and $\mathcal{Q}^i_{L}[f]$ is given by
\begin{equation}\label{interIQ}
\mathcal{Q}^i_{L}[f] = \frac{\Delta t_n}{2}\Big[f(X^i_{t_{n,0}},t_{n,L})+2\sum_{\ell=1}^{L-1}f(X^i_{t_{n,\ell}},t_{n,L}-t_{n,\ell})+f(X_{t_{n,L}}^i,0)\Big],
\end{equation}
with $\Delta t_n = t_{n,\ell}-t_{n,\ell-1} = \frac{t_n}{N}$, with $1\leq \ell\leq N$. 
Finally, we can obtain the numerical solution via
\begin{equation}\label{psi_def}
u_\ast(\bm{x}_j,t_n) =
\frac{1}{M}\sum_{i=1}^M \Psi^\ast_{\bm{x}_j,t_n}(f,g,u_0,X_\ell^i),
\end{equation}
where $\Psi^\ast_{\bm{x}_j,t_n}(\cdot)$ are defined in \eqref{phidef1} and \eqref{phidef2}.
\subsection{Space-time spectral interpolation}\label{subsect:ST spectral interpolation}
 Let $L_l(x)$, $x\in (-1,1)$ be the standard
Legendre polynomial of degree $l.$ The shifted Legendre polynomial
$\tilde{L}_{l}(t)$, $t\in (0,T)$ is defined by
$$\tilde{L}_{l}(t)=L_l\Big(\frac{2t-T}{T}\Big),\quad  l=0,1,2,\cdots.$$
The set of $\tilde{L}_{l}(t)$ is a complete $L^2(0,T)$-orthogonal system,
namely,
\begin{equation}\label{eq:2.7}
  \dint_{0}^T\tilde{L}_l(t)\tilde{L}_{m}(t)\,\d t=\dfrac{T}{2l+1}\delta_{l,m}.
\end{equation}
Denote the shifted Legendre-Gauss quadrature points and weights by $\{t_k,\omega_k\}_{k=0}^{N_t}$. We define by $h_j(t)$ the Lagrange interpolating basis polynomials associated with $\{t_k\}_{k=0}^{N_t}$. 
Using \eqref{eq:2.7}, the Lagrange interpolating basis $h_j(t)$ can be expressed by
\begin{equation}\label{Lagbas}
h_j(t)=\frac{(t-t_0)\cdots (t-t_{N_t})}{(t_j-t_0)\cdots (t_j-t_{N_t})}=\sum_{q=0}^{N_t}b_{qj}\tilde{L}_q(t), \;\;\;\text{where}\;b_{qj}=\frac{2q+1}{2}\tilde{L}_q(t_j)\omega_j.
\end{equation}
Now we turn to the space-time interpolation. To this end, we denote $\mathcal{I}_{N_t}$ be the interpolation operator associated with the grids $\{t_k\}_{k=0}^{N_t}$. To match the singularity, we employ the generalized Lagrange interpolation $\ina$ in space and standard Lagrange interpolation $\mathcal{I}_{N_t}$ in time, that is,
\begin{equation}\label{ujext}
\ina\mathcal{I}_{N_t}u(x,t) = \sum_{i=0}^{N_x}\sum_{j=0}^{N_t}u(x_i,t_j)l_i^{(\frac{\alpha}{2},\frac{\alpha}{2})}(x)h_j(t).
\end{equation}
Using \eqref{interpolationex} and \eqref{Lagbas}, we can rewrite \eqref{ujext} as 
\begin{equation}\label{1xx}
\Ia\mathcal{I}_{N_t}u(x,t)  
= \sum_{p=0}^{N_x}\sum_{q=0}^{N_t}\hat{u}_{pq}\mathcal{J}_p^{(\frac{\alpha}{2},\frac{\alpha}{2})}(x)\tilde{L}_q(t),\;\;\text{where}\;\hat{u}_{pq} = \sum_{i=0}^{N_x}\sum_{j=0}^{N_t} u(x_i,t_j) c_{p i} b_{q j}.
\end{equation}
with $c_{pi}$ and $b_{q j}$ are defined in \eqref{interpolationex} and \eqref{Lagbas}, respectively.
We can derive from \eqref{1xx} and \eqref{derpol} that
\begin{equation}\label{fracdifxx}
\begin{split}
(-\Delta)^{\frac{\alpha}{2}}\Ia\mathcal{I}_{N_t}u(x,t)
=\sum_{p=0}^{N_x}\sum_{q=0}^{{N_t}}\hat{u}_{pq}(-\Delta)^{\frac{\alpha}{2}} \mathcal{J}_p^{(\frac{\alpha}{2},\frac{\alpha}{2})}(x)\tilde{L}_q(t)
= \sum_{p=0}^{N_x} \sum_{q=0}^{{N_t}} \tilde{u}_{pq} P_p^{(\frac{\alpha}{2},\frac{\alpha}{2})}(x)\tilde{L}_q(t),
\end{split}
\end{equation}
where the connection coefficients is given by $$\tilde{u}_{pq} = \sum_{i=0}^{N_x} \sum_{j=0}^{{N_t}} \frac{\Gamma(p+\alpha+1)}{p!}u(x_i,t_j)c_{p i} b_{q j}.$$
Recall the derivative of Legendre polynomials \cite[(3.176b)]{Shen2011} and use the chain rule, we arrive at 
\begin{equation*}
\partial_t\tilde{L}_n(t)=\sum_{\substack{q=0 \\ q+n \text { odd }}}^{n-1}\frac{2(2q+1)}{T}\tilde{L}_q(t),
\end{equation*}
which together with \eqref{1xx} yields
\begin{equation}\label{dertt1}
\begin{split}
\partial_t \Ia\mathcal{I}_{N_t}u(x,t)
&=\sum_{p=0}^{N_x}\sum_{n=0}^{{N_t}}\hat{u}_{pn}\mathcal{J}_p^{(\frac{\alpha}{2},\frac{\alpha}{2})}(x)\partial_t\tilde{L}_n(t) 
= \sum_{p=0}^{N_x} \sum_{q=0}^{{N_t-1}} \check{u}_{pq}\mathcal{J}_p^{(\frac{\alpha}{2},\frac{\alpha}{2})}(x)\tilde{L}_q(t),
\end{split}
\end{equation}
where 
\begin{equation}\label{derttcoef}
\check{u}_{pq}=\sum_{\substack{n=q \\ n+q \text { odd }}}^{N_t}\hat{u}_{pn}\frac{2(2q+1)}{T} = \sum_{\substack{n=q \\ n+q \text { odd }}}^{N_t}\sum_{i=0}^{N_x}\sum_{j=0}^{N_t} u(x_i,t_j) c_{p i} b_{n j}\frac{2(2q+1)}{T}.
\end{equation}
\subsection{Implementation of space-time spectral Monte Carlo algorithm}\label{algormxt}
In this section, we restrict our attention to the following Parabolic equation with $\alpha\in(0,2]$ and homogeneous Dirichlet boundary conditions in one dimension, i.e., $\Omega=(-1,1)$:
\begin{equation}\label{ufg_1}
\begin{cases}
\partial_t u(x,t)+(-\Delta)^{\frac{\alpha}2}u(x,t)=f(x,t),\;&(x,t) \in \ \Omega \times (0,\infty) ,\\[4pt]
u(x,t)=0,\quad & (x,t) \in \  \Omega^c \times (0,\infty) ,\\[4pt]
u(x,0)=u_0(x),\quad &x \in \Omega.
\end{cases}
\end{equation}
As with the Poisson equation in section \ref{sect2}, the limiting case of \eqref{ufg_1} when $\alpha\rightarrow2$ reduces to the classical Parabolic equation, where the main difference lies in the global boundary condition $\Omega^c$ replaced by the local one $\partial\Omega$. 
Using an argument similar to that of Poisson equation, we can easily deal with nonhomogeneous Dirichlet boundary conditions. 

Now, we describe the numerical implementation of the highly accurate stochastic iteration algorithm for \eqref{ufg_1}. To summarize, the detailed implementation involves the following steps:

\smallskip
\underline{\bf Step 1:}\, Using the modified walk-on-spheres method to compute the initial iteration solution of \eqref{ufg_1} denoted by $u^{(1)}_\ast(x,t)$ at Jacobi-Gauss points $\{x_j:=x^{(\frac{\alpha}2,\frac{\alpha}2)}_j\}_{j=0}^{N_x}$ and the shifted Legendre-Gauss points $\{t_n\}_{n=0}^{N_t}$ over $(0,T)$.
\begin{equation*}
u^{(1)}_\ast(x_j,t_n)=\frac{1}{M}\sum_{m=1}^M\begin{cases}
 u_0(X^{i}_{t_{n,L}})+\mathcal{Q}^i_{L}[f],\quad & \text{if}\; X^i_{n,\ell}\in \Omega,\;\forall 1\leq \ell\leq N,\\[4pt]
g(X^{i}_{t_{n,L}},t_{L})+\mathcal{Q}^i_{L}[f],\quad & \text{if}\; X^i_{n,\ell}\in \Omega,\;\forall 1\leq \ell\leq p<N\; \&\;X^i_{n,p+1}\notin \Omega,
\end{cases}
\end{equation*}
where $M$ denotes the number of paths, $L$ and $\mathcal{Q}^i_{L}[f]$ are defined in 
\eqref{mast} and \eqref{interIQ}, respectively.

\smallskip
\underline{\bf Step 2:}\, Compute the two derivatives of $u^{(1)}_\ast(x,t)$ from initial data $\{u^{(1)}_\ast(x_j,t_n)\}_{j=0,n=0}^{N_x,N_t}$, i.e., $ (-\Delta)^{\frac{\alpha}2}u^{(1)}_\ast(x,t)$ and $\partial_tu^{(1)}_\ast(x,t)$. 
More precisely, we can construct the space-time interpolation \eqref{ujext} from initial data $\{u^{(1)}_\ast(x_j,t_n)\}$, and write $u^{(1)}_\ast(x,t)=\sum_{i=0}^{N_x}\sum_{j=0}^{N_t}u^{(1)}_\ast(x_i,t_j)l_i^{(\frac{\alpha}{2},\frac{\alpha}{2})}(x)h_j(t)$. Thus, by \eqref{fracdifxx} and \eqref{dertt1}, the two derivative of $u^{(1)}$ can be computed by
\begin{equation*} 
\begin{split}
(-\Delta)^{\frac{\alpha}{2}}u^{(1)}_\ast(x,t)= \sum_{p=0}^{N_x} \sum_{q=0}^{{N_t}} \tilde{u}_{pq}^{(1)} P_p^{(\frac{\alpha}{2},\frac{\alpha}{2})}(x)\tilde{L}_q(t),\;\;\;\partial_t u^{(1)}_\ast(x,t)=\sum_{p=0}^{N_x} \sum_{q=0}^{{N_t-1}}\check{u}_{pq}^{(1)}\mathcal{J}_p^{(\frac{\alpha}{2},\frac{\alpha}{2})}(x)\tilde{L}_q(t),
\end{split}
\end{equation*}
where the coefficients are given by
\begin{equation*}
    \tilde{u}^{(1)}_{pq} = \sum_{i=0}^{N_x} \sum_{j=0}^{{N_t}} \frac{\Gamma(p+\alpha+1)}{p!}u^{(1)}_\ast(x_i,t_j)c_{p i} b_{q j},\;\;\;
    \check{u}_{pq}^{(1)}=\sum_{\substack{n=q \\ n+q \text { odd }}}^{N_t}\sum_{i=0}^{N_x}\sum_{j=0}^{N_t} u^{(1)}_\ast(x_i,t_j) c_{p i} b_{n j}\frac{2(2q+1)}{T}.
\end{equation*}

\underline{\bf Step 3:}\, Using the walk-on-spheres method to compute the residual function $\varepsilon = u-u^{(1)}_\ast$ from the following 
\begin{equation}\label{cvmxt}
\begin{cases}
\partial_t\varepsilon+ (-\Delta)^{\frac{\alpha}2}\varepsilon = f^{(1)}, & \text { in } \ \Omega \times (0,\infty) ,\\[4pt]
\varepsilon = 0, & \text { on }\ \Omega^c \times (0,\infty) ,\\[4pt]
\varepsilon|_{t=0} = 0, & \text { in } \ \Omega,
\end{cases}
\end{equation}
where the updated right hand side source function $f^{(1)} = f-\partial_tu^{(1)}_\ast - (-\Delta)^{\frac{\alpha}2} u^{(1)}_\ast.$  Similarly, through the Feynman-Kac formula and walk-on-spheres method, we have the numerical solution $\varepsilon^{(1)}_\ast$ of \eqref{cvmxt}
\begin{equation*}
\varepsilon^{(1)}_\ast(x_j,t_n)=\frac{1}{M}\sum_{m=1}^M\mathcal{Q}^i_{L}[f^{(1)}].
\end{equation*}

\underline{\bf Step 4:}\, Update the solution by $u^{(2)}_\ast(x,t)=u^{(1)}_\ast(x,t)+\varepsilon^{(1)}_\ast(x,t).$ Clearly, we can expect that the updated solution $u^{(2)}_\ast$ is more accurate than the previous one $u^{(1)}_\ast$.

In actual computation,  we repeat Steps 2-4 until the maximum absolute difference between the two
solutions $u^{(k-1)}_\ast$ and $u^{(k)}_\ast$ is less than the desired tolerance or the machine accuracy. 
\subsection{The error estimate for ST-SMC}
\setcounter{lem}{0} \setcounter{thm}{0}  \setcounter{rem}{0}

In this section, we shall analyze and characterize the convergence of the space-time spectral Monte Carlo iteration algorithm, where we aim to estimate the error bounds in the sense that
\begin{equation}
E^\infty_{k,T}:=\max_{0\leq i \leq N_x,0\leq j \leq N_t}\big|\mathbb{E}(u^{(k)}_\ast(x_i,t_j)-u(x_i,t_j))\big|.
\end{equation}
where $\{x_i\}_{i=0}^{N_x}$ denote the Jacobi-Gauss points, $\{t_j\}_{j=0}^{N_t}$ denote the Legendre-Gauss points.

\begin{lem}\label{lem4.2}
Let $\alpha \in (0,2]$ and denote $\tilde u(x,t):=\oaa(x)u(x,t)$. Suppose that $\partial_x^{r-1}\tilde{u}(x,t)$ and $\partial_t^{s-1}\tilde{u}(x,t)$ are absolutely continuous on $\Lambda\times I$ for some $r,s\geq 1$, and the variations ${\rm Var}[\partial_x^{r}\tilde{u}]<\infty$ and ${\rm Var}[\partial_t^{s}\tilde{u}]<\infty$, then
\begin{equation}\label{inter_space-timestate}
\begin{aligned}
\big\| \Ia\mathcal{I}_{N_t}  u -  u \big\|_{L^{\infty}}\leq cN_x^{\frac{\alpha}{2}-\frac12-r}N^{\frac{1}2}_t + cN_t^{-s},
\end{aligned}
\end{equation}
where $c$ is a positive constant independent of $r$, $s$, $N_x$, and $N_t$.
\end{lem}
\begin{proof}
Since the Legendre polynomials can be viewed as a special case of the Jacobi polynomials with $\alpha = \beta =0$, it follows that
\begin{equation}\label{hjbound}
\sum^{N_t}_{j=0}|h_j(t)| = \mathcal{O}(N_t^{\frac{1}{2}}).\\
\nonumber
\end{equation}
Then, we obtain from \eqref{estimate_interpolation}, \eqref{hjbound}, and \eqref{finftynorm} with $\gamma=0$ that
\begin{equation*}
\begin{aligned}
\big\| \Ia\mathcal{I}_{N_t}  u -  u \big\|_{L^{\infty}}
&\leq \big\| \mathcal{I}_{N_t}(\Ia u - u) \big\|_{L^{\infty}} + \big\| \mathcal{I}_{N_t}  u - u \big\|_{L^{\infty}}\\
&\leq  N_t^{\frac{1}{2}}\big\| \Ia  u - u \big\|_{L^{\infty}}+ \big\| \mathcal{I}_{N_t}  u - u \big\|_{L^{\infty}} \\
&\leq cN_x^{\frac{\alpha}{2}-\frac12-r}N_t^{\frac{1}{2}}+ cN_t^{-s}.
\end{aligned}
\end{equation*}
This ends the proof.
\end{proof}

Now we prove the error bound of $u(x,t)-u_\ast(x,t)$. As with \eqref{iteration_formula}, we can rewrite the $u^{(k)}_\ast(x,t)$ in terms of $u^{(k-1)}_\ast(x,t)$ as 
\begin{small}
\begin{equation*}\label{iteration_formula1}
\begin{aligned}
u^{(k)}_\ast(x_i,t_j) = u^{(k-1)}_\ast(x_i,t_j)+ \frac{1}{M}\sum_{m=1}^{M}{\Psi}^\ast_{x_{i},t_{j},m}\big(\mathcal{A}(u-u^{(k-1)}_\ast),(u-u^{(k-1)}_\ast)|_{\Omega^c},(u-u^{(k-1)}_\ast)|_{t=0},X_t\big),
\end{aligned}   
\end{equation*}
\end{small}
where the operator $\mathcal{A} = \partial_t + (-\Delta)^{\frac{\alpha}2}$. In view of \eqref{fkluxt} and \eqref{psi_def}, we further introduce the error functional 
 \begin{equation}
 e(f,x,t) =\Psi_{x,t}(f,g,u_0,X^\alpha_t)-\Psi^\ast_{x,t}(f,g,u_0,X^i_t)= Q_{L}^i[f] - \int_{0}^{t \wedge \tau_{\Omega}^{\bx}}f(X^{\alpha}_{s},s)\,{\rm d}s.
 \end{equation}

Now we state the convergence result for the bias.
\setcounter{theorem}{0}
\begin{theorem}\label{thm4.1}
Let $\alpha \in (0,2)$, $u(x,t)$ be the solution of \eqref{ufg} and $u^{(k)}_\ast(x,t)$ be the numerical solution of SMC described in section {\rm\ref{algormxt}}. 
Denote $\tilde u(x,t):=\oaa(x)u(x,t)$, suppose that $\partial_x^{r-1}\tilde{u}(x,t)$ and $\partial_t^{s-1}\tilde{u}(x,t)$ are absolutely continuous on $\Omega\times I$ for some $r,s\geq 1$, and the variations ${\rm Var}[\partial_x^{r}\tilde{u}]<\infty$ and ${\rm Var}[\partial_t^{s}\tilde{u}]<\infty$, then we have 
\begin{equation}\label{eqn:thm4.1_1}
E_{k,T}^{\infty} \leq \widetilde{\rho}E_{k-1,T}^{\infty} + cN_x^{\frac{\alpha}{2}-\frac12-r}N_t^{\frac{1}{2}}+ cN_t^{-s} +\max_{i,j}\Big|\mathbb{E}\big[e(\mathcal{A}(u-\mathcal{I}_{N_x}^\frac{\alpha}{2}\mathcal{I}_{N_t}u),x_i,t_j)\big]\Big|,
\end{equation}    
where $0\leq i \leq N_x,0\leq j \leq  N_t$, $c$ is a positive constant independent of $N_x$, $N_t$, $k$,  and $\widetilde{\rho}$ is given by
\begin{equation*}
\widetilde{\rho} = N_xN_t\max_{i,j}\Big|\mathbb{E}[e(\mathcal{A}(l_n^{\frac{\alpha}{2}}(x)h_p(t)),x_i,t_j)]\Big|.  
\end{equation*}
For any fixed $N_x, N_t$ and sufficiently large $M$, one has $\widetilde{\rho} < 1$. Thus, the convergence of $\{E^\infty_{k,T}\}_k$, is geometric at rate $\widetilde{\rho}$, up to a threshold equal to
\begin{equation}\label{eqn:thm4.1_2} 
E_{k,T}^{\infty} \leq \frac{1}{1-\widetilde{\rho}}\Big[cN_x^{\frac{\alpha}{2}-\frac12-r}N_t^{\frac{1}{2}}+ cN_t^{-s} +\max_{i,j}\Big|\mathbb{E}\big[e(\mathcal{A}(u-\mathcal{I}_{N_x}^\frac{\alpha}{2}\mathcal{I}_{N_t}u),h,x_i,t_j)\big]\Big|\Big].
\end{equation}    
\end{theorem}
\begin{proof}
Since the proof for parabolic-type equations can be derived in a similar manner to that of Poisson-type equations, as shown in \eqref{err_E_k-1} in the proof of Theorem \ref{thm3.3}, we obtain that
\begin{small}
\begin{equation*}
\begin{aligned}
\mathbb{E}^{(k-1)}[u^{(k)}_\ast(x_i,t_j)-u(x_i,t_j)]  &= [u-\mathcal{I}_{N_x}^\frac{\alpha}{2}\mathcal{I}_{N_t}u](x_i,t_j) + \Big|\mathbb{E}\Big[e\big(\mathcal{A}(u-\mathcal{I}_{N_x}^\frac{\alpha}{2}\mathcal{I}_{N_t}u),x_i,t_j\big)\Big]\Big| \\
&+ \sum_{n=0}^{N_x}\sum_{p=0}^{N_t}(u(x_n,t_p)-u^{(k-1)}_\ast(x_n,t_p))\mathbb{E}^{(k-1)}\Big[e(\mathcal{A}(l_n^{\frac{\alpha}{2}}(x)h_p(t)),x_i,t_j)\Big].
\end{aligned}
\end{equation*}
\end{small}
Take absolute values on both sides of the above equation, we find that
\begin{equation*}
\begin{aligned}
&\big|\mathbb{E}^{(k-1)}[u^{(k)}_\ast(x_i,t_j)-u(x_i,t_j)]\big| \leq  \Big|[u-\mathcal{I}_{N_x}^\frac{\alpha}{2}\mathcal{I}_{N_t}u](x_i,t_j)\Big| +\Big|\mathbb{E}\Big[e\big(\mathcal{A}(u-\mathcal{I}_{N_x}^\frac{\alpha}{2}\mathcal{I}_{N_t}u),h,x_i,t_j\big)\Big]\Big| \\
&\quad +cN_xN_t\max_{n,p}\big|u(x_n,t_p)-u^{(k-1)}_\ast(x_n,t_p)\big|\big|\mathbb{E}^{(k-1)}[e(\mathcal{A}(l_n^{\frac{\alpha}{2}}(x)h_p(t)),h,x_i,t_j)]\big|.
\end{aligned}
\end{equation*}
Taking the expectation and pointwise maximum norm on both sides, and using the law
of total expectation \eqref{eyy} and the fact that $\big|\mathbb{E}[\mathbb{E}^{(k-1)}[\cdot]]\big| \leq \mathbb{E}\big|\mathbb{E}^{(k-1)}[\cdot]\big|$ leads to desired result \eqref{eqn:thm4.1_1}.
Clearly, \eqref{eqn:thm4.1_2} is a direct consequence of \eqref{eqn:thm4.1_1}. This completes the proof.
\end{proof}

\begin{remark}{\em
In the above Theorem, the convergence factor $\widetilde{\rho}$  depends on the error functional $\max_{i,j}\Big|\mathbb{E}[e(\mathcal{A}(l_n^{\frac{\alpha}{2}}(x)h_p(t)),x_i,t_j)]\Big|$. Similar to the proof of \eqref{EB1}, we can obtain $$\max_{i,j}\Big|\mathbb{E}[e(\mathcal{A}(l_n^{\frac{\alpha}{2}}(x)h_p(t)),x_i,t_j)]\Big|\leq cM^{-\frac12},$$ which implies that as $M\rightarrow \infty$ with fixed $N_x$ and $N_t$, we have $\widetilde{\rho}=cN_xN_tM^{-\frac12}\rightarrow0$.}
\end{remark}

\section{Numerical results} \label{sect5}\setcounter{lem}{0} \setcounter{thm}{0}  \setcounter{rem}{0} 	
In this section, we present ample numerical results to illustrate the efficiency and accuracy of our spectral and space-time spectral Monte Carlo methods for computing solutions of poisson problem \eqref{uf} and parabolic problem \eqref{ufg}, respectively. 
In what follows, we use the following discrete $L^\infty$-error for function $u(x)$ and $u(x,t)$ are
\begin{equation*}
E^\infty_N= \max_{i}\Big|\mathbb{E}\big(u(\hat x_i)-u_\star(\hat x_i)\big)\Big|, \;\;\;E^\infty_{N,T}= \max_{i,j}\Big|\mathbb{E}\big(u(\hat x_i,\hat t_j)-u_\star(\hat x_i,\hat t_j)\big)\Big|,
\end{equation*}
where $u_\star$ denotes the numerical solution, and $\{\hat{x}_i\}_{i=0}^{N_x}$ and $\{\hat{t}_j\}_{j=0}^{N_t}$ denote the quadrature points over interval $(-1,1)$ and $(0,T)$, respectively. 

\begin{exa}\label{Ex:1}{\rm (Accuracy test for Poisson equation with exact solution)} 
We first consider fractional Poisson equation \eqref{uf} with global homogeneous boundary condition $g(x)=0$ on $\Omega^c$ and the following exact solution 
\begin{equation}\label{solu1}
 u_1(x) = (1-x^2)^{\frac{\alpha}{2}}(x^2+x+1),\;\;\;  u_2(x) = (1-x^2)^{\frac{\alpha}{2}}\sin(x).
\end{equation}
The source function $f(x)$ associated with $u_1(x)$ can be obtained easily by using the derivative relation \eqref{derpol}. While for $u_2(x)$, the exact source function is unknown, so we can approximate the exact solution $u_2(x)$ by using GJFs and again employ the derivative relation \eqref{derpol} to obtain the corresponding source function with machine precision.
\end{exa}

It is evident that both solutions exhibit singularities at the boundary $\partial\Omega$, and we employ the proposed spectral Monte Carlo method to calculate numerical solutions. In Fig.\ref{fig:1.1}(left), we plot the discrete maximum error, in semi-log scale, at random points (i.e.,$E^\infty_N$) for $u_1(x)$ against iteration number $K$, for which we take the number of simulation paths $M=50$, the degree of GJFs to be $N_x=2$, and various fractional power $\alpha=0.4,0.8,1.2,1.6,2$. We observe that even for non-smooth solution the error decays exponentially with respect to the iteration number $K$ as expected. We were pleasantly surprised to find that the number of iterations $K$ required to achieve spectral accuracy (about $10^{-16}$) decreases as $\alpha$ decreases. 
The discrete maximum point-wise errors are displayed in Fig.\ref{fig:1.1}(right), and once again, the exponential convergence is observed. Here we take the degree of GJFs to be $N_x=2$, fractional power $\alpha =0.4$, and the number of paths $M=2, 10, 100$. This finding suggests that by increasing the number of simulation paths $M$, we can achieve spectral accuracy with fewer iteration steps. 
\begin{figure}[htbp]
\centering  
\subfigure{
\includegraphics[width=0.46\textwidth]{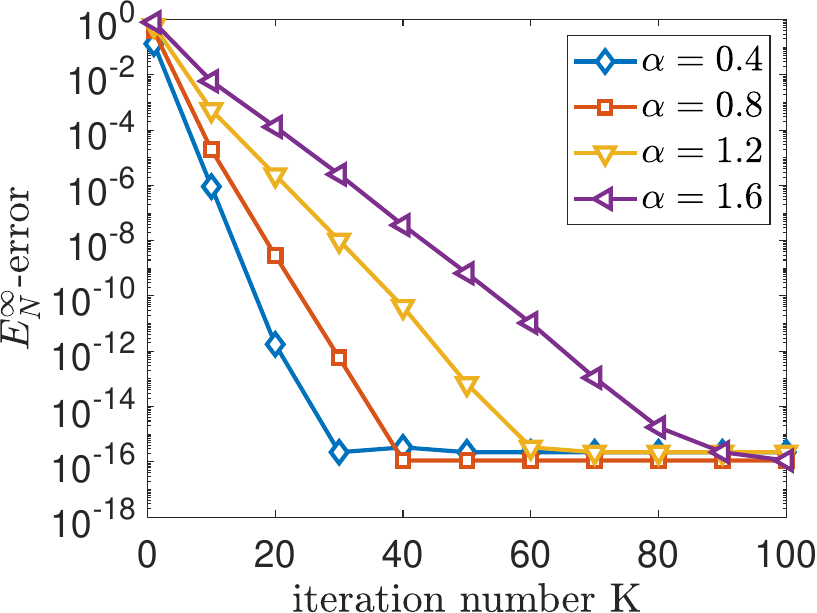}}
\subfigure{
\includegraphics[width=0.44\textwidth]{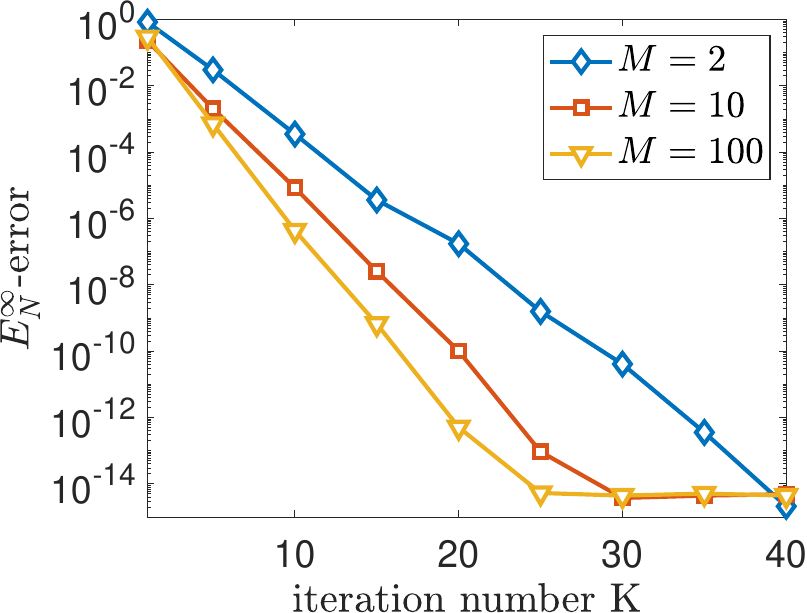}}
\caption{The discrete $E_N^\infty$-errors for $u_1(x)$ in \eqref{solu1} against the iteration number $K$. Left: $M = 50, N_x =2$; Right: $\alpha =0.4, N_x = 2$. }
\label{fig:1.1}
\end{figure}

In Fig.\ref{fig:1.2}(left), we plot discrete maximum error, in semi-log scale, for $u_2(x)$ against the degree of GJFs $N_x$. Here we take $M=10$ and various $\alpha=0.4,0.8,1.2,1.6,2$. The expected exponential convergence is also observed in this setting. 
Moreover, we plot the discrete maximum point-wise errors in Fig.\ref{fig:1.2}(right), where the error decays exponentially with respect to iteration number $K$. In this computation, we take $\alpha =0.4$ and the number of paths to be $M=2, 10, 100$. We observe similar behaviors as in the case of $u_1(x)$, where increasing the number of simulation paths $M$ allows us to achieve spectral accuracy with fewer iteration steps.

\begin{figure}[htbp]
\centering  
\subfigure{
\label{Fig.1.1}
\includegraphics[width=0.45\textwidth]{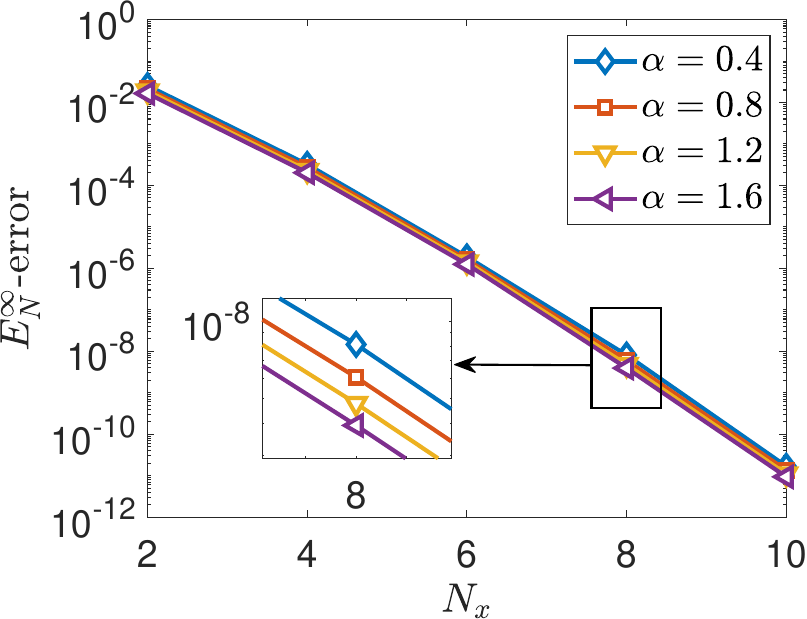}}
\subfigure{
\label{Fig.1.2}
\includegraphics[width=0.45\textwidth]{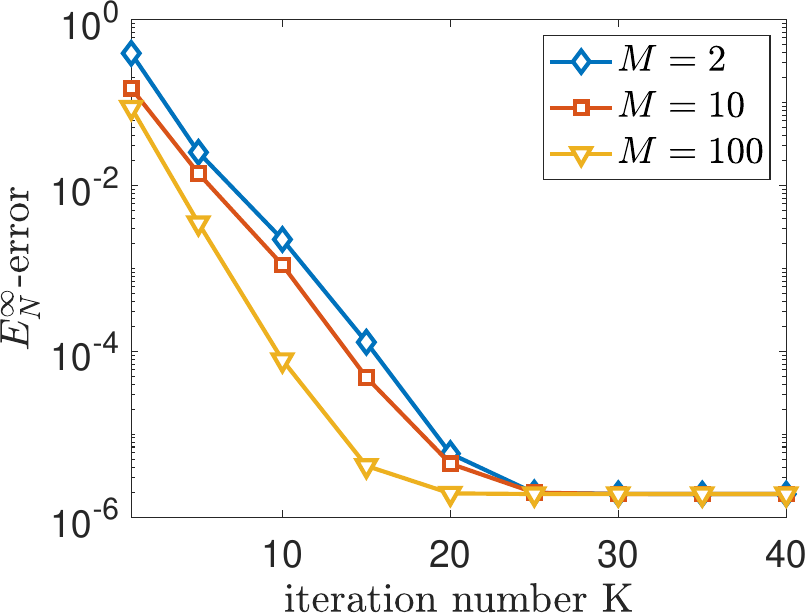}}
\caption{The discrete $E_N^\infty$-errors for $u_2(x)$ in \eqref{solu1}. Left: with fixed $M = 10$; Right: take $\alpha = 0.4, N_x =6$. }
\label{fig:1.2}
\end{figure}

\begin{exa}\label{ex1}{\rm(Given source term $f(x)$ for Poisson equation)} Next, we consider Poisson equation \eqref{uf} with the global homogeneous boundary condition $g(x) = 0$ on $\Omega^c$, and the right-hand side source term is given by $f(x)=\sin(x)$. 
\end{exa}
\begin{figure}[htbp]
\centering  
\subfigure{
\includegraphics[width=0.45\textwidth]{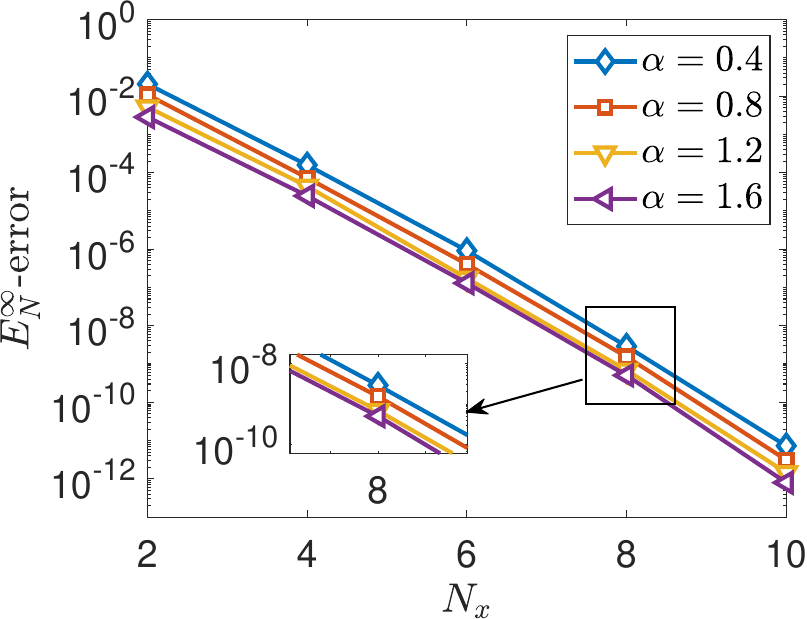}}
\subfigure{
\includegraphics[width=0.45\textwidth]{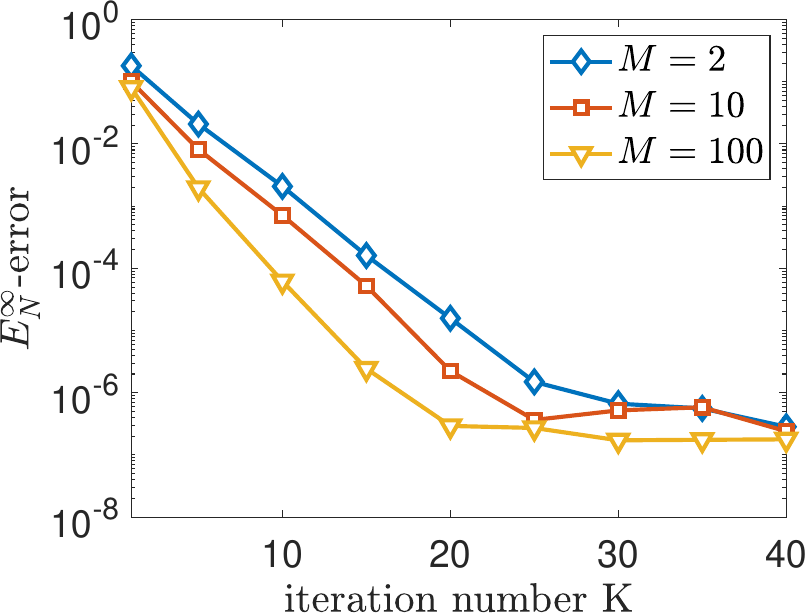}}
\caption{The discrete $E_N^\infty$-errors for example 2. Left: $M = 10$; Right: $\alpha=0.4, N_x=6$. }
\label{fig:2.1}
\end{figure}

Since the exact solution is unknown, we use the GJFs-Galerkin spectral method (cf. \cite{Chen2018}) with $N_x=100$ to obtain the numerical solution as the reference solution. 
In our simulation, we use the proposed spectral Monte Carlo method to obtain numerical solutions and take the number of simulation paths $M=10$ and various $\alpha=0.4,0.8,1.2,1.6,2$. In Fig.\,\ref{fig:2.1} (left), we present the discrete $L^\infty$-error, in semi-log scale, versus the number of modes $N_x$. This plot clearly indicates the exponential convergence of our proposed spectral Monte Carlo method for any given smooth source function. Again, on the right of Fig.\,\ref{fig:2.1}, we plot the discrete maximum point-wise errors versus iteration number $K$. We can observe a similar phenomenon of exponential convergence and a similar relationship between the number of simulation paths $M$ and the number of iterations $K$.  

\begin{exa}{\rm (Accuracy test for Parabolic equation with exact solution)} We consider parabolic equation {\rm \eqref{ufg}} with global homogeneous boundary condition $g(x,t)=0$ on $\Omega^c\times(0,T)$ and the following exact solution 
\begin{equation}\label{solupara}\begin{split}
 &u_1(x,t) = (1-x^2)^{\frac{\alpha}{2}}(x^2+x+1)\cos(t),\;\;\; 
u_2(x,t) = (1-x^2)^{\frac{\alpha}{2}}\sin(x)\cos(t),
 \end{split}\end{equation}
 where we take $\Omega\times(0,T):=(-1,1)\times(0,0.5).$
The source function $f(x,t)$ for the above solutions can be computed in a similar fashion to that of the Poisson equation. 
\end{exa}

We use the proposed space-time spectral Monte Carlo method to evaluate numerical solutions. 
In Fig.\ref{fig:3.1}(left), we depict the discrete maximum error, in semi-log scale, for $u_1(x,t)$, where we fixed the number of path $M=50$, the degree of modes $N_x=2,N_t = 6$, and fractional power $\alpha =0.4, 0.8,1.2,1.6,2$. It is seen that the error decays exponentially with respect to $K$ for the parabolic equation. 
In Fig.\ref{fig:3.1}(right), we plot maximum error, in semi-log scale versus iteration number $K$, for which we take the number of paths $M=2, 10, 100$, the degree of modes $N_x=6, N_t=6$, $\alpha=0.4$, and $T=0.5$. As with Poisson equation, it has been proven that setting a smaller fractional index $\alpha$ leads to faster convergence. Additionally,  we observe from Fig.\ref{fig:3.1} that increasing the number of simulation paths enables quicker convergence.  
\begin{figure}[htbp]
\centering  
\subfigure{
\includegraphics[width=0.45\textwidth]{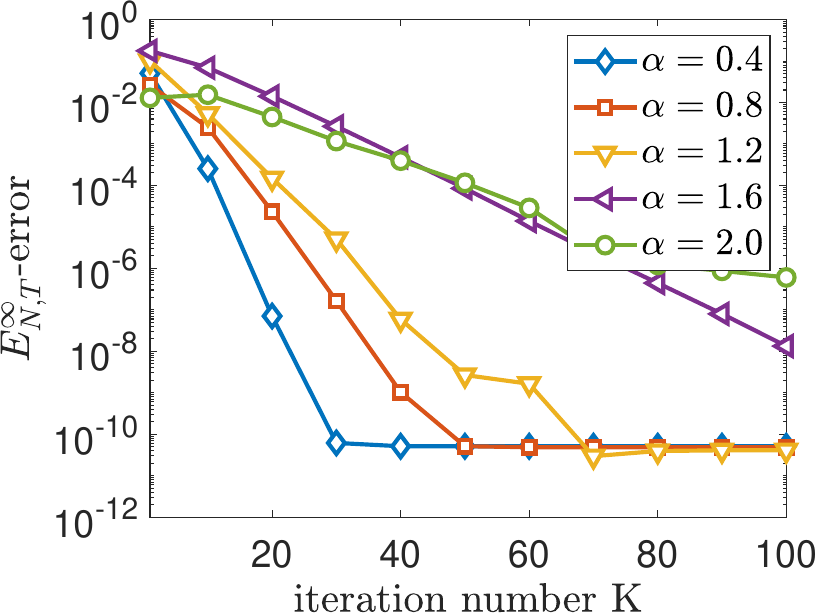}}
\subfigure{
\includegraphics[width=0.45\textwidth]{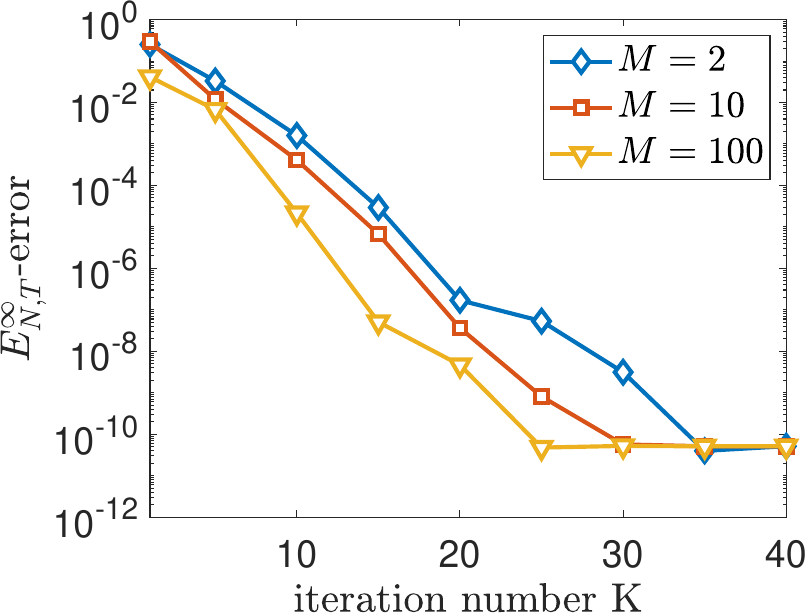}}
\caption{The discrete $E_{N,T}^\infty$-errors for $u_1(x,t)$ in \eqref{solupara}. Left: with fixed $N_x=2,N_t=6,M=50$; Right: take $N_x=6,N_t=6,\alpha=0.4$.}
\label{fig:3.1}
\end{figure}

\begin{figure}[htbp]
\centering  
\subfigure{
\includegraphics[width=0.45\textwidth]{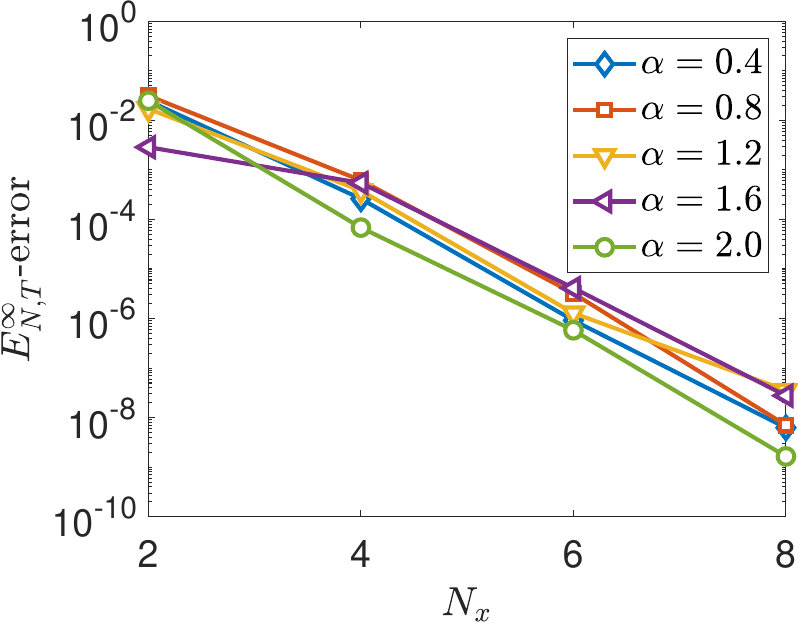}}
\subfigure{
\includegraphics[width=0.45\textwidth]{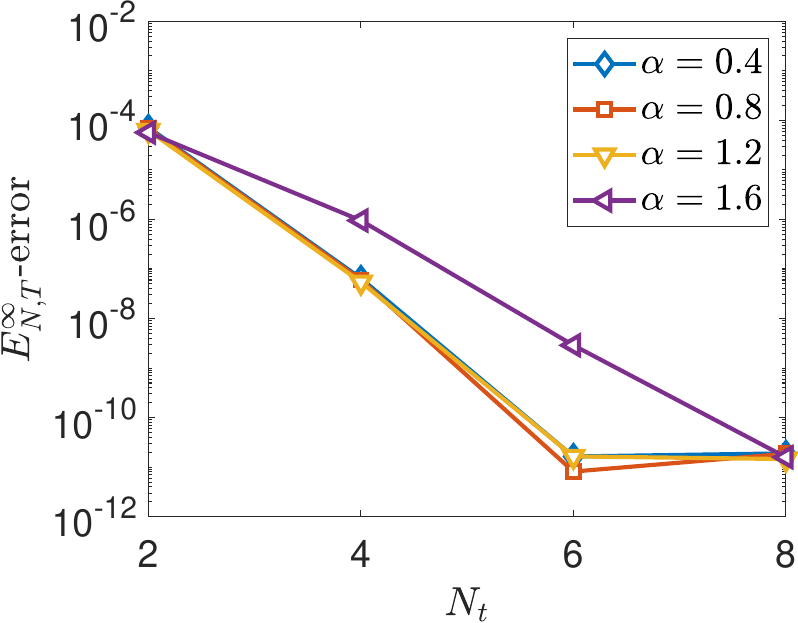}}
\caption{The discrete $E_{N,T}^\infty$-errors for $u_2(x,t)$ in \eqref{solupara}. Left:  $N_t=5,M=2$; Right: take $N_x=10,M=2$. }
\label{fig:3.2}
\end{figure} 
For $u_2(x,t)$, we take the number of path $M=2$, the degree of modes $N_x=2,4,6,8,N_t = 5$, and various $\alpha =0.4, 0.8,1.2,1.6$. In Fig.\ref{fig:3.2}, we plot the maximum error, in semi-log scale, versus the degree of spatial modes $N_x$ and temporal modes $N_t$, respectively. 
As we can see from Fig.\ref{fig:3.2}, both the numerical errors in spatial and temporal direction can achieve spectral accuracy against various $N_x$ and $N_t$. 
These numerical results demonstrate that our space-time Monte Carlo method can achieve the same level of approximation accuracy as classical spectral methods even for non-smooth solutions. The only difference is that our proposed method does not require solving linear systems of equations.
 
\begin{exa}\label{Ex:0.5}{\rm (Extended to classical Parabolic equation in three-dimensions.)} 
We consider the parabolic equation \eqref{uf} with $\alpha=2$  on the three-dimensional unit cube $\Omega = [-1, 1]^3$ with homogeneous Dirichlet boundary conditions. We choose the exact solution as
\begin{equation*}
    u(\bx, t) = \exp\Big(\sum_{i=1}^3\frac{x_i}{8}\Big)\prod_{i=1}^3(1-x_i^2)\cos(t), \quad \bx = (x_1,x_2,x_3) \in  [-1, 1]^3.
\end{equation*}
\end{exa}

In this example, we focus on the special case $\alpha=2$. Similarly, we employ Jacobi-Gauss points with index $(1,1)$ in the spatial direction and Legendre-Gauss-Radau points in the temporal direction. More precisely, the basis functions in the spatial direction are constructed in a tensor product form 
$$l_{n_1,n_2,n_3}(x_1,x_2,x_3) =l^{(1,1)}_{n_1}(x_1)l^{(1,1)}_{n_2}(x_2)l^{(1,1)}_{n_3}(x_3),\;\;\; 0\leq n_1,n_2,n_3\leq N_x,$$
where the generalized Lagrange interpolating function $l^{(1,1)}_{n_\ell}(x_\ell)$ is defined in \eqref{fracLagbasxx} along the $x_\ell$-direction, and the shifted Legendre basis functions $h_j(t)$, introduced in Section \ref{subsect:ST spectral interpolation}, are used for the temporal direction.  
For convenience, we compute the maximum error, also denoted by $E^{\infty}_{N,T}$, at randomly selected points along each spatial direction at $T=0.2$. 
We further denote by $L$ the number of iterations required to achieve convergence for a given $M$.
As expected, the numerical results presented in Fig. \ref{fig:4.1} demonstrate that the numerical error decreases as the number of spatial and temporal nodes, $N_x$ and $N_t$, increases. Moreover, appropriately increasing the number of sample paths $M$ can to some extent enhance the convergence rate, thereby reducing the number of iterations $L$ required to achieve the desired accuracy.  
Numerical results indicate that to ensure the convergence of the proposed algorithm, the number of sample paths $M$ must be sufficiently large relative to $N_x$ and $N_t$.

\begin{figure}[htbp]
\centering  
\subfigure{
\includegraphics[width=0.45\textwidth]{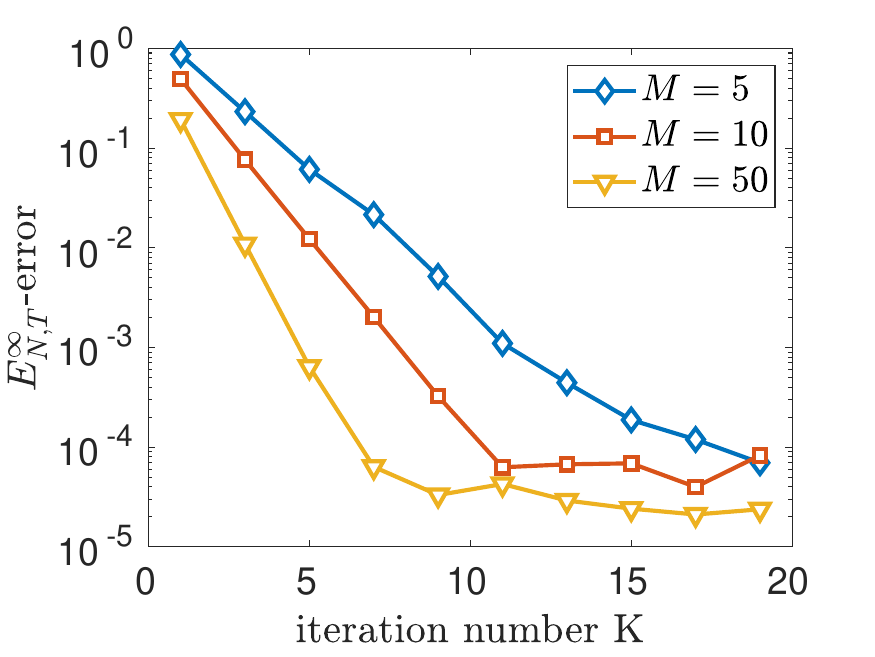}}
\subfigure{
\includegraphics[width=0.45\textwidth]{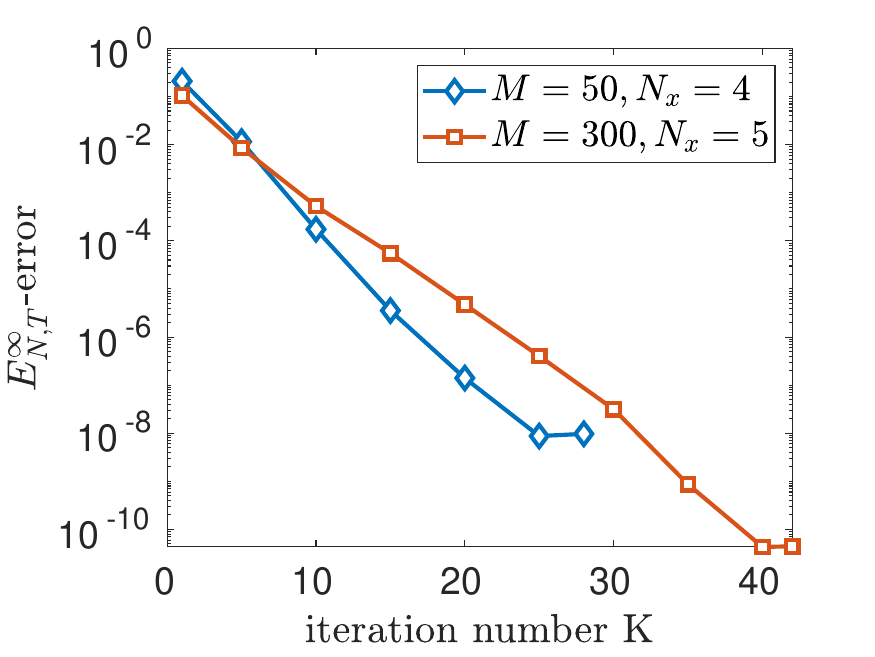}}
\caption{The discrete $E_{N,T}^\infty$-errors for $u(\bx,t)$. Left:  $N_x=2,N_t=6$; Right: take $N_t=6$. }
\label{fig:4.1}
\end{figure} 
 
\section{Conclusions}
We have proposed a new spectral Monte Carlo method for solving the Poisson equation driven by $\alpha$-stable Lévy process and successfully extended it to solve the parabolic equation driven by  $\alpha$-stable Lévy process using a space-time Monte Carlo method. Both of these methods are theoretically proven to exhibit exponential convergence, and the numerical results have further confirmed this. Unlike traditional methods, the proposed stochastic approaches bypass the need to solve linear systems and allow for parallel computations of numerical solutions at each grid point. This parallelization is particularly beneficial for parabolic equations, where simultaneous computations can be performed in both the space and time directions. 
Moreover, compared to existing works, our method has a wider applicability range, and importantly, it includes the classical methods in the limit case with $\alpha=2$, so we claim it as the unified method.  While this paper only provided numerical results in one dimension, it is worth noting that by replacing the generalized Jacobi functions with a suitable set of interpolation functions, one potential candidate is radial basis functions. The proposed methods have great potential for generalization to higher-dimensional cases, and we leave this extension as part of our future work.

\vspace{18pt}

\noindent{\bf Declarations}

\begin{itemize}
  \item {\bf Availability of data and materials:}  The datasets generated during and/or analysed during the current study are available from the corresponding author on reasonable request. 
  \item {\bf Authors' contributions:} All authors contributed to this study. The computations and the first draft were prepared by the first and second authors. All authors read and approved the final manuscript.
\item {\bf Conflict of interest statement:}   We have no conflicts of interest to disclose.
\end{itemize}

\vspace{18pt}

\bibliographystyle{siamplain}

\begin{thebibliography}{10}
\bibitem{Acosta2017}
{\sc G. Acosta, F. Bersetche, and J. Borthagaray}, {\em { A short FE implementation for a 2d homogeneous Dirichlet problem of a fractional Laplacian}}, Comput. Math. Appl., 74 (2017), pp. 784--816.

\bibitem{Alanko2013}
{\sc S. Alanko and M. Avellaneda}, {\em { Reducing variance in the numerical solution of BSDEs}}, C. R. Math. Acad. Sci. Paris, 351(2013), pp. 135--138. 

\bibitem{Andersson2019}
{\sc K. Andersson, A. Andersson and C. W. Oosterlee}, {\em { Machine learning approximation algorithms for high-dimensional fully nonlinear partial differential equations and second-order backward stochastic differential equations}}, J. Nonlinear Sci., 29 (2019), pp. 1563--1619.

\bibitem{Babuska2004}
{\sc I. Babuška, R.I. Tempone and G.E. Zouraris}, {\em { Galerkin finite element approximations of stochastic elliptic partial differential equations}}, SIAM J. Numer. Anal., 42 (2004), pp. 800--825.

\bibitem{Bally2005}
{\sc V. Bally, P.Gilles and J.Printems}, {\em {A quantization tree method for pricing and hedging multi-dimensional American options.}}, Math. Financ., 15 (2005), pp. 119--168.

\bibitem{Billaud-Friess2024}
{\sc M. Billaud-Friess, A.Macherey, A. Nouy and C. Prieur}, {\em A probabilistic reduced basis method for parameter-dependent problems}, Adv. Comput. Math., 50(2024), p.19.

\bibitem{Bird1996}
{\sc G. Bird, J. L. Lumley, and G. Berkooz}, {\em { Molecular Gas Dynamics and the Direct Simulation of Gas Flows}}, Oxford Engineering Science Series, Oxford University Press, 1994. 

\bibitem{Brockmann2006}
{\sc D.~Brockmann, L.~Hufnagel, and T.~Geisel}, {\em The scaling laws of human travel}, Nature, 439(2006), pp. 462--465.

\bibitem{Brehier2024}
{\sc C. Bréhier, D. Cohen, and J. Ulander}, {\em { Analysis of a positivity-preserving splitting scheme for some nonlinear stochastic heat equations}}, ESAIM Math. Model. Numer. Anal., 58 (2024), pp. 1317--1346.

\bibitem{Cai2011}
{\sc N. Cai and S.~ Kou}, {\em {Option Pricing Under a Mixed-Exponential Jump Diffusion Model.}}, Manage. Sci., 57 (2011), pp. 2067--2081.

\bibitem{Canuto2012}
{\sc C. Canuto, M.~ Hussaini, A. Quarteroni and T.~ Zang}, {\em {Spectral methods in fluid dynamics.}}, Springer Ser. Sci. Comput., Springer, Berlin, 2012.

\bibitem{Chassagneux2012}
{\sc J. F. Chassagneux}, {\em {Linear multistep schemes for BSDEs.}}, SIAM J. Numer. Anal., 52 (2014), pp. 2815--2836.

\bibitem{Chen2016}
{\sc S.~Chen, Z.~Mao, and H.~Li}, {\em {Generalized Jacobi functions and their applications to fractional differential equations}}, Math. Comp. 85 (2016), pp. 1603--1638.		

\bibitem{Chen2018}
{\sc L.~Chen, Z.~Mao, and H.~Li}, {\em {Jacobi-Galerkin spectral method for eigenvalue problems of Riesz fractional differential equations}}, arXiv preprint arXiv:1803.03556, (2018).		

\bibitem{Chen2010}
{\sc Y.~Chen and T.~Tang}, {\em {Convergence Analysis of the Jacobi Spectral-Collocation Methods for Volterra Integral Equations with a Weakly Singular Kernel}}, Math. Comp., 79 (2010), pp. 147--167.

\bibitem{Chopin2024}
{\sc N. Chopin and M. Gerber}, {\em {Higher-order Monte Carlo through cubic stratification}}, SIAM J. Numer. Anal., 62 (2024), pp. 229--247.

\bibitem{DeLaurentis1990}
{\sc J. DeLaurentis and L. Romero}, {\em A Monte Carlo method for Poisson's equation}, J. Comput. Phys., 90 (1990), pp. 123--140.

\bibitem{Ding2022}
{\sc C. Ding, C. Yan, X. Zeng, and W. Cai}, {\em A parallel iterative probabilistic method for mixed problems of Laplace equations with the Feynman-Kac formula of killed Brownian motions}, SIAM J. Sci. Comput. 44 (2022), pp. A3413--A3435.

\bibitem{Ding2023}
{\sc C. Ding, Y. Zhou, W. Cai, X. Zeng, and C. Yan}, {\em A path integral Monte Carlo (PIMC) method based on Feynman-Kac formula for electrical impedance tomography}, J. Comput. Phys. 476 (2023), p. 111862.

\bibitem{Du2012}
{\sc Q. Du, M. Gunzburger, R. B. Lehoucq, and K. Zhou}, {\em Analysis and approximation of nonlocal diffusion problems with volume constraints}, SIAM Rev., 54 (2012), pp. 667--696.

\bibitem{During2012}
{\sc B. Düring and M. Fournié}, {\em High-order compact finite difference scheme for option pricing in stochastic volatility models}, J. Comput. Appl. Math., 236 (2012), pp. 4462--4473.


\bibitem{Hutzenthaler2021}
{\sc W. E, M. Hutzenthaler, A. Jentzen, and T. Kruse}, {\em Multilevel Picard iterations for solving smooth semilinear parabolic heat equations}, Partial Differ. Equ. Appl., 2 (2021), pp. 1--31.

\bibitem{Elfverson2016}
{\sc D. Elfverson, F. Hellman, and A. Målqvist}, {\em { A multilevel Monte Carlo method for computing failure probabilities}}, SIAM/ASA J. Uncertain. Quantif., 4 (2016), pp. 312--330. 

\bibitem{Glasserman2004}
{\sc P. Glasserman}, {\em { Monte Carlo Methods in Financial Engineering}}, Springer, New York, 2004.  

\bibitem{Gobet2010}
{\sc E. Gobet and C. Labart}, {\em Solving BSDE with Adaptive Control Variate}, SIAM J. Numer. Anal., 48(2010), pp. 257--277.

\bibitem{Gobet2004}
{\sc E.~Gobet and S.~Maire}, {\em {A spectral Monte Carlo method for the Poisson equation}}, Mcma, 10 (2004), pp. 275--285.

\bibitem{Gobet2005}
{\sc E.~Gobet and S.~Maire}, {\em {Sequential Control Variates for Functionals of Markov Processes}}, SIAM J. Numer. Anal., 43(2005), pp. 1526--1275.

\bibitem{GobetMenozzi2010}
{\sc E. Gobet and S. Menozzi}, {\em Stopped diffusion processes: boundary corrections and overshoot}, Stoch. Proc. Appl., 120 (2010), pp. 130--162.

\bibitem{Grubb2015}
{\sc G. Grubb}, {\em Fractional laplacians on domains, a development of Hörmander's theory of $\mu$-transmission pseudodifferential operators}, Adv. Math., 268 (2015), pp. 478--528

\bibitem{Han2022}
{\sc Q. Han and S. Ji}, {\em A multi-step algorithm for BSDEs based on a predictor-corrector scheme and least-squares Monte Carlo}, Methodol. Comput. Appl., 24 (2022), pp. 2403--2426. 

\bibitem{Holmes1996}
{\sc P. Holmes, J. L. Lumley, and G. Berkooz}, {\em { Turbulence, Coherent Structures, Dynamical Systems and Symmetry}}, Cambridge Monogr. Mech., Cambridge Univ. Press, Cambridge, UK, 1996.

\bibitem{Huang2014}
{\sc Y. Huang and A. Oberman}, {\em Numerical methods for the fractional Laplacian: A finite difference-quadrature approach}, SIAM J. Numer. Anal., 52 (2014), pp. 3056--3084.

\bibitem{Jiao2023}
{\sc C.~Jiao and C.~Li}, {\em A modified walk-on-sphere method for high dimensional fractional Poisson equation}, Numer. Methods Partial Differential Eq. 39 (2023), pp. 1128--1162.

\bibitem{Karniadakis2013}
{\sc G. Karniadakis and S. Sherwin}, {\em { Spectral/hp Element Methods for Computational Fluid Dynamics}}, Oxford University Press, 2013.


\bibitem{Komori2017}
{\sc Y. Komori, D. Cohen, and K .Burrage}, {\em { Weak second order explicit exponential Runge-Kutta methods for stochastic differential equation}}, SIAM J. Sci. Comput., 39 (2017), pp. A2857--A2878.

\bibitem{Kyprianou2018}
{\sc A.~ Kyprianou, A.~Osojnik, and T.~Shardlow}, {\em {Unbiased `walk-on-spheres' Monte Carlo methods for the fractional Laplacian}}, IMA J. Numer. Anal., 38 (2018), pp. 1550--1578.
		
\bibitem{Ladd1994}
{\sc A.J.C. Ladd}, {\em { Numerical simulations of particulate suspensions via a discretized Boltzmann equation: Part 2. Numerical results}}, J. Fluid Mech., 271 (1994), pp. 311--339. 

\bibitem{LeMaitre2010}
{\sc O. P. Le Maitre and O. M. Knio}, {\em { Spectral Methods for Uncertainty Quantification: With Applications to Computational Fluid Dynamics}}, Springer Ser. Sci. Comput., Springer, New York, 2010.

\bibitem{Longo2021}
{\sc M. Longo, S. Mishra, T.K. Rusch, and C. Schwab}, {\em {  Higher-order quasi-Monte Carlo training of deep neural networks}}, SIAM J. Sci. Comput., 43 (2021), pp. A3938--A3966. 

\bibitem{Lord2014}
{\sc G. J. Lord, C. E. Powell, and T. Shardlow}, {\em { An Introduction to Computational Stochastic PDEs}}, Cambridge Univ. Press, 2014.

\bibitem{Lyons2016}
{\sc R. Lyons and Y. Peres}, {\em { Probability on Trees and Networks}}, Camb. Ser. Stat. Probab. Math., Cambridge Univ. Press, New York, 2016. 

\bibitem{Maire2015}
{\sc S. Maire and M. Simon}, {\em A partially reflecting random walk on spheres algorithm for electrical impedance tomography}, J. Comput. Phys., 303 (2015), pp. 413--430.

\bibitem{Mao2016}
{\sc  Z. Mao, S. Chen, J. Shen}, {\em Efficient and accurate spectral method using generalized Jacobi functions for solving Riesz fractional differential equations}, Appl. Numer. Math. 106 (2016), pp. 165--181.

\bibitem{Mascagni2004}
{\sc M.~Mascagni and N.A.~Simonov}, {\em {Monte Carlo methods for calculating some physical properties of large molecules}}, SIAM J. Sci. Comput., 26 (2004), pp. 339--357.
	
\bibitem{Muller1956}
{\sc M. Muller}, {\em {Some continuous Monte Carlo methods for the Dirichlet Problem}}, Ann. Math. Stat. 27 (1956), pp. 569--589.

\bibitem{Nanbu1980}
{\sc K. Nanbu}, {\em {Direct simulation scheme derived from the Boltzmann equation. I. Monocomponent gases}}, J. Phys. Soc. Jpn., 49 (1980), pp. 2042--2049.

\bibitem{Nezza2012}
{\sc E.~D. Nezza, G.~Palatucci, and E.~Valdinoci}, {\em {Hitchhiker's guide to the fractional Sobolev spaces}}, Bull. Sci. Math., 136 (2012), pp. 521--573.

\bibitem{Shao2015}
{\sc S. Shao and J. M. Sellier}, {\em {Comparison of deterministic and stochastic methods for time-dependent Wigner simulations}}, J. Comput. Phys., 300 (2015), pp. 167--185.

\bibitem{Shao2020}
{\sc S.~Shao and Y.~Xiong}, {\em {Branching random walk solutions to the Wigner equation}}, SIAM J. Numer. Anal. 58 (2020), pp. 2589--2608.

\bibitem{Shen2006}
{\sc J. Shen and T. Tang}, {\em { Spectral and high-order methods with applications}}, Science Press, Beijing, 2006.

\bibitem{Shen2011}
{\sc J. Shen, T. Tang, and L.-L. Wang}, {\em {Spectral Methods: Algorithms, Analysis and Applications}}, Springer Ser. Comput. Math. 41, Springer, Berlin, 2011.

\bibitem{Sheng2023}
{\sc C.~Sheng, B.~Su, and C.~Xu}, {\em {Efficient Monte Carlo method for integral fractional Laplacian in multiple dimensions}}, SIAM J. Numer. Anal., 61 (2023), pp. 2035--2061.

\bibitem{Sheng2024}
{\sc C.~Sheng, B.~Su, and C.~Xu}, {\em {An implicit-explicit Monte Carlo method for semi-linear PDEs driven by $\alpha$-stable Lévy process and its error estimates}}, submitted.

\bibitem{Su2023}
{\sc B.~Su, C.~Sheng, and C.~Xu}, {\em {A new `walk on spheres' type method for fractional diffusion equation in high dimensions based on the Feynman-Kac formulas}}, Appl. Math. Lett., 141 (2023), p. 108597.

\bibitem{Szego1975}
{\sc G. Szegö}, {\em {Orthogonal Polynomials}}, Orthogonal Polynomials, American Mathematical Society, 5th ed., Providence, RI, 1975.

\bibitem{Takahashi2022}
{\sc A. Takahashi, Y. Tsuchida and T. Yamada}, {\em {A new efficient approximation scheme for solving high-dimensional semilinear PDEs: control variate method for Deep BSDE solver}}, J. Comput. Phys., 454 (2022), p. 110956.

\bibitem{Wu2020}
{\sc L. Wu and W. Chen}, {\em { The sliding methods for the fractional p-Laplacian}}, Adv. Math. 361 (2020), p. 106933.

\bibitem{Xiang2016}
{\sc S. Xiang}, 
{\em { On interpolation approximation: convergence rates for polynomial interpolation for functions of limited regularity}} SIAM J. Numer. Anal. 54 (2016), pp.  2081--2113.

\bibitem{Xiu2005}
{\sc D. Xiu and J. S. Hesthaven}, {\em { High-order collocation methods for differential equations with random inputs}}, SIAM J. Sci. Comput., 27 (2005), pp. 1118--1139.

\bibitem{Xiu2010}
{\sc D. Xiu}, {\em { Numerical Methods for Stochastic Computations: A Spectral Method Approach}}, Springer Ser. Math. Eng., Springer, Cham, 2018.
D. Xiu, Princeton Univ. Press, 2010. 

\bibitem{Xiu2002}
{\sc D. Xiu and G.~ Karniadakis}, {\em {The Wiener–Askey polynomial chaos for stochastic partial differential equations}}, SIAM J. Sci. Comput., 24 (2002), pp. 619--644.

\bibitem{Yan2021}
{\sc C. Yan, W. Cai, and X. Zeng}, {\em {A highly scalable boundary integral equation and walk-on-spheres (BIE-WOS) method for the Laplace equation with Dirichlet data}}, Commun. Comput. Phys. 29 (2021), pp. 1446--1468.

\bibitem{Zhao2014}
{\sc W. Zhao, Y. Fu and T. Zhou}, {\em {New kinds of high-order multi-step schemes for forward backward stochastic differential equations}}, SIAM J. Sci. Comput. 36 (2014),  pp. 1731--1751.

\bibitem{Zhao2010}
{\sc W. Zhao, G. Zhang, and L. Ju}, {\em {A Stable Multistep Scheme for Solving Backward Stochastic Differential Equations}}, SIAM J. Numer. Anal., 48 (2010), pp. 1369--1394.
\end{thebibliography}

\end{document}